\documentclass[a4paper]{amsart}
\usepackage{amsmath,amsthm,amsfonts,amssymb,color,hyperref,verbatim,appendix}

\numberwithin{equation}{section}
\theoremstyle{plain}

\newtheorem{theorem}{\sc Theorem}[section]

\newtheorem{corollary}[theorem]{\sc Corollary}
\newtheorem{definition}[theorem]{\sc Definition}

\newtheorem{lemma}[theorem]{\sc Lemma}
\newtheorem{proposition}[theorem]{\sc Proposition}

\theoremstyle{remark}

\newtheorem{remark}[theorem]{\sc Remark}

\newcommand{\one}{{{\rm 1\mkern-1.5mu}\!{\rm I}}}
\newcommand{\be}{\begin{equation}}
\newcommand{\ee}{\end{equation}}

\def\t{\mathbf{t}}

\begin{document}

\title[The Stochastic Encounter-Mating Model]{The Stochastic Encounter-Mating Model}

\author{Onur G\"un}
\address{Onur G\"un, Weierstrass Institute, Mohrenstrasse 39, 10117 Berlin, Germany.}
\email{Onur.Guen@wias-berlin.de}

\author{Atilla Yilmaz}
\address{Atilla Yilmaz, Department of Mathematics, Ko\c{c} University, Sar\i yer, Istanbul 34450, Turkey.}
\email{atillayilmaz@ku.edu.tr}

\date{Revised on July 31, 2016.}


\begin{abstract}
	
We propose a new model of permanent monogamous pair formation in zoological populations with multiple types of females and males. According to this model, animals randomly encounter members of the opposite sex at their so-called firing times to form temporary pairs which then become permanent if mating happens. Given the distributions of the firing times and the mating preferences upon encounter, we analyze the contingency table of permanent pair types in three cases: (i) definite mating upon encounter; (ii) Poisson firing times; and (iii) Bernoulli firing times. In the first case, the contingency table has a multiple hypergeometric distribution which implies panmixia. The other two cases generalize the encounter-mating models of Gimelfarb (1988) who gives conditions that he conjectures to be sufficient for panmixia. We formulate adaptations of his conditions and prove that they not only characterize panmixia but also allow us to reduce the model to the first case by changing its underlying parameters. Finally, when there are only two types of females and males, we provide a full characterization of panmixia, homogamy and heterogamy.

\end{abstract}

\maketitle

\section{Introduction}

\subsection{Overview}\label{overview}

In a 1988 paper \cite{Gim88a}, Gimelfarb studies two models of permanent monogamous pair formation in zoological populations comprised of $k\ge2$ types of females and males. In both models, males randomly encounter females to form temporary pairs which then become permanent if mating takes place. In the first model (called individual encounter), only one random single male initiates an encounter at any time, whereas in the second one (called mass encounter), all single males do. The probability of mating upon encounter depends on the types of the female and the male, and is denoted by $p_{ij}$ for type-$ij$ pairs. Gimelfarb is interested in how the mating preference matrix $P := (p_{ij})$ determines the mating pattern, i.e., the contingency table of pair types once all animals mate. He focuses on the concept of panmixia which refers to female and male types being uncorrelated in the expected mating pattern, and proposes a sufficient condition for it in terms of $P$ for each of the two models.

Motivated by Gimelfarb's work, we introduce a new model of permanent monogamous pair formation, called the stochastic encounter-mating (SEM) model, where both females and males initiate encounters at their so-called firing times which are determined by independent point processes whose distributions depend on the sex and the type of the animals. We start by considering the case of definite mating upon encounter (i.e., $p_{ij} = 1$) with general firing time distributions, and prove the following results:
\begin{itemize}
\item [(i)] the mating pattern has a multiple hypergeometric distribution which implies panmixia; 
\item [(ii)] the mating pattern is independent of the firing times.
\end{itemize}
Then, we provide formulas for the distribution and the expectation of the contingency table of pair types at any time.

When the point processes are Poisson with intensity $\alpha_i$ (resp.\ $\beta_j$) for type-$i$ females (resp.\ type-$j$ males), the SEM model is a continuous-time Markov chain and its infinitesimal generator depends on $P = (p_{ij})$, $\alpha_1,\ldots,\alpha_k$ and $\beta_1,\ldots,\beta_k$ only through $\Pi = (\pi_{ij})$ with $\pi_{ij} := p_{ij}(\alpha_i + \beta_j)$. We use $\Pi$ to formulate the so-called Poisson fine balance condition which we show to be necessary for panmixia. Moreover, when this condition is satisfied, the underlying parameters can be changed (if necessary) to yield definite mating upon encounter, and our results for that case carry over. In particular, the Poisson fine balance condition characterizes panmixia. Similarly, when the point processes are Bernoulli with success probability $\alpha_i$ (resp.\ $\beta_j$) for type-$i$ females (resp.\ type-$j$ males), the SEM model is a discrete-time Markov chain and its transition kernel involves $\Pi = (\pi_{ij})$ with $\pi_{ij} := p_{ij}(\alpha_i + \beta_j - \alpha_i\beta_j)$. Analogous to the previous case, we use $\Pi$ to formulate the so-called Bernoulli fine balance condition which characterizes panmixia and allows us to reduce the model to definite mating upon encounter.

Gimelfarb's individual (resp.\ mass) encounter model corresponds to SEM when the point processes are Poisson (resp.\ Bernoulli) with $\alpha_i = 0$ and $\beta_j = 1$, and his sufficient condition for panmixia implies the Poisson (resp.\ Bernoulli) fine balance condition. Thus, in this paper we not only unify and generalize Gimelfarb's models but also rigorously prove stronger versions of his panmixia conjectures. Moreover, for both of the Poisson and Bernoulli cases, when $k=2$, we establish a trichotomy for $\Pi$ that fully characterizes heterogamy/panmixia/homogamy, that is, negative/zero/positive correlation of same type females and males in the expected mating pattern.

The rest of this paper is structured as follows. Section \ref{themodel} gives a precise description of the SEM model. Section \ref{lagalug} defines the key concepts of panmixia, homogamy and heterogamy. Section \ref{literatur} surveys the previous results on similar models (including a detailed account of Gimelfarb's aforementioned work \cite{Gim88a}), and provides biological background, motivation and references. Section \ref{ourresults} summarizes our results which are grouped under the following headings: (i) general firing times; (ii) Poisson firing times; and (iii) Bernoulli firing times. The full statements and the proofs of these results are presented in Sections \ref{generalsection}, \ref{poissonsection} and \ref{bernoullisection}, respectively. Finally, Section \ref{conicsection} is devoted to some concluding remarks, observations and open problems.

\subsection{The SEM model}\label{themodel}

Consider an animal species comprised of $k\ge2$ types of females and males forming permanent monogamous heterosexual pairs. Label the types $1,\ldots,k$. Take a population consisting of $x_i\ge0$ type-$i$ females and $y_j\ge0$ type-$j$ males for every $i,j\in\{1,\ldots,k\}$, such that
\begin{equation}\label{destur}
n := x_1+\cdots +x_k = y_1+\cdots+y_k\ge1,
\end{equation}
i.e., there are an equal number of females and males in the population. Denote the set of females (resp.\ males) by $Z_f = \{\varphi_1,\ldots,\varphi_n\}$ (resp.\ $Z_m = \{\mu_1,\ldots,\mu_n\}$), and the whole population by $Z := Z_f\cup Z_m$.

Associated to each $\zeta\in Z$, there is a simple point process $N(\zeta)$ on the real half-line $[0,\infty)$ with time points $(\tau_s(\zeta))_{s\ge1}$. These point processes are mutually independent and the distribution of $N(\zeta)$ depends on the sex and the type of $\zeta$. For the sake of convenience, we will refer to $(\tau_s(\zeta))_{s\ge1}$ as the firing times of $\zeta$.
At these times, $\zeta$ initiates an encounter (i.e., forms a temporary pair) with a random member of the opposite sex. Whether such an encounter results in mating (i.e., the formation of a permanent pair) is also random, and the conditional probability of this event (upon encounter) depends on the types of the female and the male involved. Because of these two stages of encounter and mating, we will call this model of permanent pair formation the {\it stochastic encounter-mating} (SEM) model.

Let us now give a precise description of the model. We start with imposing certain conditions on the firing times.
\begin{definition}\label{topgaz}
A two-dimensional array $\mathbf{t}:= (t_s(\zeta))_{s\ge1,\zeta\in Z}$ of positive extended real numbers is said to be a proper family of firing times if it satisfies the following conditions:
\begin{itemize}
\item [(a)] $t_s(\zeta)\le t_{s+1}(\zeta)$ and the inequality is strict whenever $t_s(\zeta)<\infty$;
\item [(b)] $t_1(\zeta)<t_2(\zeta)<\cdots<\infty$ for every $\zeta\in Z_f$ or for every $\zeta\in Z_m$.
\end{itemize}
The set of all proper families of firing times is denoted by $\Phi$.
\end{definition}
\noindent Our technical assumptions (see Section \ref{technicalpoints}) on the distributions of the firing times will ensure that $\tau := (\tau_s(\zeta))_{s\ge1,\zeta\in Z}\in\Phi$ almost surely.

For every $t\ge0$, let $S_f(t)$ (resp.\ $S_m(t)$) denote the set of single females (resp.\ males) at that time. We will refer to $S(t) := S_f(t)\cup S_m(t)$ as the singles' pool. Initially, all of the animals are single, i.e., $S(0) = Z$. Given any $\mathbf{t}\in\Phi$, the distinct times at which at least one single animal fires are recursively defined as
\begin{equation}\label{guzel}
t_r^*= t_r^*(\t) := \min\{t_s(\zeta): s\ge1, \zeta\in S(t_{r-1}^*), t_s(\zeta) > t_{r-1}^*\}
\end{equation}
for $r\ge1$, where $t_0^* := 0$ as a convention. These are the times of the so-called firing rounds.
The encounter and mating stages at the $r$th firing round are as follows.

{\it Stage I: Encounter.} If exactly one element of $S(t_{r-1}^*)$, say a female, fires at $t_r^*$, then it samples a male from $S_m(t_{r-1}^*)$ uniformly at random and initiates an encounter with it. If two or more elements of $S(t_{r-1}^*)$ fire at $t_r^*$, then each of them samples a single member of the opposite sex uniformly at random and without replacement. Here is exactly how this is done: The elements of $S(t_{r-1}^*)$ that fire at $t_r^*$ are ordered in an arbitrary way. (For the sake of definiteness, order them with respect to their labels and let the females go first if any. However, as we will see in Remark \ref{bahaii}, this order does not matter.) Assume without loss of generality that the first one is a female $\varphi_a\in S_f(t_{r-1}^*)$ and it samples a male $\mu_b\in S_m(t_{r-1}^*)$. Then, neither $\varphi_a$ nor $\mu_b$ can be sampled by the subsequent elements of $S(t_{r-1}^*)$ that fire at $t_r^*$. Moreover, even if $\mu_b$ fires at $t_r^*$, it is not allowed to sample a female from $S_f(t_{r-1}^*)$ when its turn comes, because it is already in a temporary pair with $\varphi_a$. This procedure continues until all the elements of $S(t_{r-1}^*)$ that fire at $t_r^*$ are in a temporary pair. The collection of these pairs is denoted by $(\Delta\mathcal{L})'(t_r^*)$. This choice of notation will become clear in the next two paragraphs.

{\it Stage II: Mating.} After the encounter stage at the $r$th firing round is completed, independent Bernoulli random variables are assigned to the pairs in $(\Delta\mathcal{L})'(t_r^*)$ to determine whether they become permanent (i.e., if the courtship results in mating). The probability of mating for each such type-$ij$ pair (i.e., if the female and the male are of type-$i$ and type-$j$, respectively) is equal to some $p_{ij}>0$. The $k\times k$ matrix $$P := \left(p_{ij}\right)$$ is called the {\it mating preference matrix} of the species. The collection of permanent pairs formed at the $r$th firing round is denoted by $\Delta\mathcal{L}(t_r^*)\subset(\Delta\mathcal{L})'(t_r^*)$. The singles' pool $S(t_r^*)$ is obtained by removing from $S(t_{r-1}^*)$ the animals in the pairs constituting $\Delta\mathcal{L}(t_r^*)$. Note that the animals in the pairs in $(\Delta\mathcal{L})'(t_r^*)\setminus\Delta\mathcal{L}(t_r^*)$, i.e., the ones that have formed temporary but not permanent pairs at the $r$th firing round, remain in the singles' pool.

This two-stage procedure is iteratively carried out at $t_1^*,t_2^*,\ldots$ and it naturally stops at
$$T := \min\{t>0: S(t) = \emptyset\} = \min\{t_r^*: r\ge1, S(t_r^*) = \emptyset\},$$
i.e., when the singles' pool is depleted. Since $p_{ij}>0$ and every female or every male fires infinitely many times in $[0,\infty)$ by Definition \ref{topgaz}, the terminal time $T$ is almost surely finite. For any $t\in[0,T]$, let
$$\mathcal{L}(t) := \bigcup_{\substack{r\ge1:\\t_r^*\le t}}\Delta\mathcal{L}(t_r^*)$$
be the collection (or unordered list) of permanent pairs formed by time $t$. Similarly, for any $i,j\in\{1,\ldots,k\}$, let $Q_{ij}(t)$ be the number of permanent type-$ij$ pairs formed by time $t$. The set-valued process $\mathcal{L}(\cdot)$ and the $k\times k$ matrix-valued process $Q(\cdot) := \left(Q_{ij}(\cdot)\right)$ are called the pair-list process and the pair-type process, respectively. The terminal value $Q(T)$ of the latter is referred to as the {\it mating pattern} of the population. For the sake of convenience, we set $\mathcal{L}(t) = \mathcal{L}(T)$ and $Q(t) = Q(T)$ for every $t>T$.

Let $\nu$ be the probability measure that the firing times $\tau = (\tau_s(\zeta))_{s\ge1,\zeta\in Z}$ induce on $(\Phi,\mathcal{F})$, where $\mathcal{F}$ is the Borel $\sigma$-algebra corresponding to the product topology. Given any $\mathbf{t}\in\Phi$, let $\mathbb{P}^\mathbf{t}$ be the probability measure that (i) the sequence $\left((\Delta\mathcal{L})'(t_r^*)\right)_{r\ge1}$ of temporary pairs and (ii) the pair-list process $\mathcal{L}(\cdot)$ induce on some appropriate measurable space $(\Omega,\mathcal{B})$ that is common for all $\mathbf{t}\in\Phi$. We prefer not to explicitly define $(\Omega,\mathcal{B})$ since we will never directly refer to it. Finally, define $\mathbb{P}$ as the semi-direct product measure of $\nu$ and $\mathbb{P}^\mathbf{t}$ on $(\Phi\times\Omega,\mathcal{F}\otimes\mathcal{B})$, i.e., $\mathbb{P}(d\mathbf{t},d\omega) := \nu(d\mathbf{t})\mathbb{P}^\mathbf{t}(d\omega)$. 
Note that $\mathbb{P}^\mathbf{t}(\cdot) = \mathbb{P}(\cdot\,|\,\tau = \mathbf{t})$ for $\mathbf{t}\in\mathrm{supp}(\nu)$. We write $\mathbb{E}^\t$ (resp.\ $\mathbb{E}$) to denote expectation under $\mathbb{P}^\t$ (resp.\ $\mathbb{P}$).

The firing time distributions and the mating preference matrix are jointly referred to as
the {\it encounter-mating (EM) law}. In this paper, we will consider the following EM laws:
\begin{itemize}
\item [(Def)] general firing times \& definite mating upon encounter (Section \ref{generalsection});
\item [(Poi)] Poisson firing times \& general mating preferences (Section \ref{poissonsection});
\item [(Ber)] Bernoulli firing times \& general mating preferences (Section \ref{bernoullisection}).
\end{itemize}
As we outline in Section \ref{literatur}, cases (Poi) and (Ber) are generalizations of models that have been previously introduced in the literature, whereas our main purpose in analyzing case (Def) is to unify and clarify the other two.

%
%

\subsection{Panmixia, homogamy and heterogamy}\label{lagalug}

Given the EM law of the species and a population with the number of animals of each sex and type satisfying (\ref{destur}), we would like to analyze the distribution of the pair-type process $Q(\cdot)$. In our analysis, for any $i,j\in\{1,\ldots,k\}$, we will frequently refer to the following quantities:
\begin{align}
u_{ij}(t;x_1,\ldots,x_k;y_1,\ldots,y_k) &:= \mathbb{E}\left[Q_{ij}(t)\right]\quad\mbox{for any $t\ge0$};\label{tozert}\\
u^*_{ij}(x_1,\ldots,x_k;y_1,\ldots,y_k) &:= \mathbb{E}\left[Q_{ij}(T)\right]\quad\mbox{for the terminal time $T$; and}\label{bayankut}\\
u^*_{ij}(x_1,\ldots,x_k;y_1,\ldots,y_k\,|\,\t) &:= \mathbb{E}^{\t}[Q_{ij}(T)]\quad\mbox{for any $\t\in\Phi$.}\label{nboz}
\end{align}
Among these, (\ref{bayankut}) will play a pivotal role.

We start with two observations. First, since all the animals mate by the terminal time $T$, the mating pattern $Q(T)$ is a random $k\times k$ contingency table whose $i$th row sum and $j$th column sum are equal to $x_i$ and $y_j$, respectively. In particular, it has
\begin{equation}\label{degfree}
k\cdot k - (k + k) + 1 = (k-1)^2
\end{equation}
degrees of freedom. Second, if there is definite mating upon encounter, i.e., if $p_{ij} = 1$ for every $i,j\in\{1,\ldots,k\}$, then all of the $n!$ possible terminal pair-lists should be equally likely. (We will prove this later.) In particular, each of the $x_i$ type-$i$ females forms a permanent pair with a type-$j$ male with probability $y_j/n$. This motivates the following definition.

\begin{definition}
A population is said to be panmictic if $$u^*_{ij}(x_1,\ldots,x_k;y_1,\ldots,y_k) = \frac{x_iy_j}{n}$$
for every $i,j\in\{1,\ldots,k\}$, where $x_i\ge0$ (resp.\ $y_j\ge0$) is the number of type-$i$ females (resp.\ type-$j$ males), satisfying (\ref{destur}). The species is said to be panmictic if every population of animals from the species is panmictic.
\end{definition}

\noindent In words, panmixia for a species refers to having (on average) zero correlation between female and male types in permanent pairs in any population.

To complement the concept of panmixia, homogamy (resp.\ heterogamy) is defined as females and males of similar types having positive (resp.\ negative) correlation in $\mathbb{E}\left[Q(T)\right]$. However, in order to make these definitions precise, one has to choose an appropriate metric on the set $\{1,\ldots,k\}$ of types. Also, note that the definitions of panmixia, homogamy and heterogamy are not {\it a priori} collectively exhaustive even for a fixed population. For example, type-$1$ females and males can be positively correlated while type-$2$ females and males are negatively correlated. However, when $k=2$, i.e., there are only two types of females and males, these potential issues are ruled out. Indeed, there is a unique choice of metric on the set $\{1,2\}$, and $Q(T)$ has only one degree of freedom by (\ref{degfree}). In particular, finding $u^*_{11}(x_1,x_2;y_1,y_2)$ is sufficient for determining the other three. Therefore, the following concise definitions make sense.

\begin{definition}\label{homhet}
For $k=2$, a population is said to be
\begin{align*}
\mbox{homogamous if}\quad u^*_{11}(x_1,x_2;y_1,y_2) &> \frac{x_1y_1}{n},\quad\mbox{and}\\
\mbox{heterogamous if}\quad u^*_{11}(x_1,x_2;y_1,y_2) &< \frac{x_1y_1}{n},
\end{align*}
where $x_i \ge 1$ (resp.\ $y_j \ge 1$) is the number of type-$i$ females (resp.\ type-$j$ males), such that $n = x_1 + x_2 = y_1 + y_2$.
The species is said to be homogamous (resp.\ heterogamous) if every population of animals from the species is homogamous (resp.\ heterogamous).
\end{definition}

\begin{remark}
In the definition above, we have not allowed $x_1x_2y_1y_2$ to be $0$ because, in that case, the mating pattern $Q(T)$ is deterministic and we trivially have $$u^*_{11}(x_1,x_2;y_1,y_2) = \frac{x_1y_1}{n}.$$
\end{remark}

\subsection{Previous results and related literature}\label{literatur}

As mentioned in Section \ref{overview}, Gimelfarb \cite{Gim88a} introduces two models of encounter-mating (EM) for  permanent monogamous pair formation. Both of these models are in discrete time, and they differ from each other only in the firing times of the animals. In the first one, called individual EM, exactly one uniformly sampled single male (and no female) fires at each $t\in\mathbb{N} = \{1,2,3,\ldots\}$, whereas in the second one, called mass EM, all single males (and no females) fire at each $t\in\mathbb{N}$. Given the firing times, the encounter and mating stages at each firing round of these models are as described in Section \ref{themodel}. Therefore, they are special cases of our SEM model. Indeed, for individual EM, the pair-type process is the discrete-time process embedded in $Q(\cdot)$ when $\{N(\zeta)\}_{\zeta\in Z}$ are Poisson processes with intensity $0$ (resp.\ $1$) for each $\zeta\in Z_f$ (resp.\ $\zeta\in Z_m$). Similarly, mass EM corresponds to the SEM model when $\{N(\zeta)\}_{\zeta\in Z}$ are Bernoulli processes with success probability $0$ (resp.\ $1$) for each $\zeta\in Z_f$ (resp.\ $\zeta\in Z_m$).

In order to simplify the analysis, Gimelfarb replaces all of the quantities such as $Q(\cdot)$ with their expectations. He says that, because of the law of large numbers (LLN), this is a good approximation when $n$ is large, but he does not rigorously justify this claim. He defines the concepts of panmixia, homogamy and heterogamy, but only in an asymptotic sense as $n\to\infty$. He then asserts that the species is (asymptotically) panmictic whenever
\begin{align}
p_{ij} &= \bar\alpha_i + \bar\beta_j\quad\mbox{for some $\bar\alpha_1,\ldots,\bar\alpha_k,\bar\beta_1,\ldots,\bar\beta_k$ in the individual EM model, and}\label{iecon}\\
p_{ij} &= 1 - \bar\gamma_i\bar\delta_j\quad\mbox{for some $\bar\gamma_1,\ldots,\bar\gamma_k,\bar\delta_1,\ldots,\bar\delta_k$ in the mass EM model.}\label{mecon}
\end{align}
He says that he was unable to prove that (\ref{iecon}) is a sufficient condition for panmixia, and instead provides numerical evidence to back up this claim. In contrast, he does give an argument to show that (\ref{mecon}) is a sufficient condition for panmixia. However, this argument is not rigorous because of his underlying LLN approximation (i.e., replacing quantities with their expectations).

Gimelfarb draws two main conclusions from the sufficient conditions (\ref{iecon}) and (\ref{mecon}) for panmixia. First, for either model, given a population, there is a many-to-one correspondence between mating preference matrices and expected mating patterns. 
This is a very important point; not only theoretically, but also practically. Indeed, the mating pattern of a population can be observed in its habitat whereas the mating preferences have to be determined by laboratory experiments which are relatively costly. It is therefore tempting to try to infer the latter from the former, but Gimelfarb concludes that this inverse problem cannot be solved. His second conclusion is that knowing just the mating preference matrix is not enough for predicting the mating pattern. We need to also know whether we have, say, individual EM or mass EM. On a related note, Gimelfarb mentions that these two are of course not the only possible models. They are the extreme ones in some sense, and various intermediate cases can be considered, too. We will come back to this point later in Section \ref{conicsection}.

Gimelfarb was not the first author to propose an earlier version of the SEM model. Indeed, what he calls individual EM was previously introduced by Romney \cite{Rom71} in the context of anthropology to model marriages in a community. Mosteller \cite{Mos68} analyzed Romney's model and gave recursive equations for the expected mating pattern, which can be iteratively solved when $n$ is sufficiently small. What is particularly important about Mosteller's approach is that he did not make the LLN approximation of Gimelfarb, as the latter replaces the SEM model with a deterministic one. On this note, prior to Gimelfarb, such a deterministic EM model was proposed by Taylor \cite{Tay75} who gave a system of ordinary differential equations describing the evolution of $Q(t)$. However, he could provide only numerical solutions for this system. Also, neither Mosteller nor Taylor mentioned panmixia in their papers.

Panmixia is one of the fundamental concepts in population genetics. In particular, it is one of the main assumptions underlying the Hardy-Weinberg law which states that genotype frequencies remain constant in a population to which no evolutionary force acts on, see, e.g., \cite{Ewe04} for details. In the literature, panmixia is also referred to as random mating. Gimelfarb \cite{Gim88a} points out that the latter term is rather misleading especially for a bottom-up approach (i.e., from mating preferences to mating patterns) such as in the SEM model. Indeed, random mating suggests that the animals do not have any preferences, i.e., there exists a constant $p\in(0,1]$ such that $p_{ij} = p$ for every $i,j\in\{1,\ldots,k\}$. However, as we have already said, panmixia is about the expected mating pattern and there are many mating preference matrices that give rise to it. 

In the cases of homogamy and heterogamy, the genotype frequencies might differ greatly from the ones predicted by the Hardy-Weinberg law, see \cite[Chapter 4]{Gil98} and the references therein. The prevalence of these cases is studied in detail across species in \cite{JiaBolKir13}. Also, for the effect of homogamy and heterogamy on the genetical evolution of a finite population, see e.g., \cite{Eth08}, which is inspired by \cite{Der92}. In these works, the terms positive and negative assortative mating are used to mean homogamy and heterogamy, respectively. Again, one should keep in mind that these terms refer to the expected mating pattern and not to the mating preferences. Indeed, 
(say) type-$1$ females in a homogamous population need not necessarily prefer mating with type-$1$ males. See Section \ref{conicsection} for a class of examples.

In the population dynamics literature, most models of pair formation assume that the females unilaterally accept or reject the males. This assumption is generally realistic (see \cite{And94} for the empirical aspects of sexual selection), and its various consequences have been studied in, e.g., \cite{Kir82}, \cite{Lan81}, and recently in \cite{Alex1}. Therefore, on one hand, the fact that the SEM model removes this assumption might seem unnecessary. On the other hand, not specifying which sex accepts or rejects the other one extends the scope of the model and makes it potentially suitable for two-sided matching problems that also have applications outside of biology. Such problems are typically studied using game theory, see the survey paper \cite{Ber00} and the series of papers \cite{Alp99, Alp05} regarding assortative mating. Moreover, in contrast to Gimelfarb's EM models, allowing both females and males to fire at prescribed rates not only makes the SEM model more versatile, but also introduces degrees of freedom in the EM law which then can be exploited to yield exact formulas under certain conditions. This is precisely what we will do in our analysis.

Having mentioned some of the advantages of the SEM model, we should point out that it clearly has limitations, too, as it is a model for permanent monogamous pair formation without births, deaths or offsprings. In his aforementioned paper \cite{Gim88a}, Gimelfarb also considers models where either one or both sexes are allowed to be polygamous. These models turn out to be much easier to analyze and less interesting in their behavior. Moreover, in another paper \cite{Gim88b}, he introduces a simple model of pair formation that allows separation of pairs and analyzes it using information theoretical concepts. Regarding models incorporating births, deaths and offsprings, we refer the reader to the recent paper by Hadeler \cite{Had12} and the references therein. From a biological point of view, the fact that our model does not include these factors can be partially justified by assuming that everything takes place in one breeding season.



Finally, as the aforementioned work of Romney \cite{Rom71} attests, the scope of the SEM model is not restricted to non-human animals. Indeed, the demographic, cultural, and technological changes of the last 10,000 years did not preclude the potential for natural and sexual selection in our species \cite{Alex2}. Determining the mating patterns of human populations has crucial applications, e.g., modelling how sexually transmitted diseases spread, see \cite{DieHad88}. However, we can attempt to answer such questions only after incorporating separations and/or polygamy into our model.

\subsection{Summary of our results}\label{ourresults}

Before giving the precise statements of our results on the SEM model in Sections \ref{generalsection}, \ref{poissonsection} and \ref{bernoullisection}, which require technical assumptions and further notation, we will summarize them below. While doing so, we will provide motivation for them, and explain how they are related to each other as well as to the previous results in the literature.

\subsubsection{General firing times}

Recall that, while describing the SEM model in Section \ref{themodel}, we have made certain assumptions regarding the joint distribution of the point processes $\{N(\zeta)\}_{\zeta\in Z}$ which give the firing times $\tau = (\tau_s(\zeta))_{s\ge1}$ of the animals. There, we have also said that we will make further technical assumptions to ensure that $\tau\in\Phi$ (see Definition \ref{topgaz}) almost surely. We start Section \ref{generalsection} by listing all of these assumptions which hold throughout the paper.

The sampling procedure at the encounter stage (see Section \ref{themodel}) is the most complicated aspect of the SEM model. In Section \ref{alternativerepsection}, we give an alternative representation of the model which amounts to decomposing the encounter stage at any firing round into two steps: at the first one, we pair up all of the single animals uniformly and without replacement; and at the second one, we discard the pairs that do not have any animals that have fired in that round. We make repeated use of this representation in our arguments.

In Section \ref{defmatesection}, we consider the SEM model under the assumption of definite mating upon encounter, i.e., we take $p_{ij} = 1$ for every $i,j\in\{1,\ldots,k\}$. In this special case, it is possible to give a detailed analysis and provide exact formulas for (\ref{tozert}), (\ref{bayankut}) and (\ref{nboz}). This is due to the fact that the mating stages at the firing rounds are trivial, i.e., the two-stage structure of the model is reduced to one. Our first result is Theorem \ref{bekirog} which says that the terminal pair-list $\mathcal{L}(T)$ is uniformly distributed under $\mathbb{P}^\t$ for every $\t\in\Phi$. In particular, $\mathcal{L}(T)$ and $\tau = (\tau_s(\zeta))_{s\ge1}$ are independent under $\mathbb{P}$. This is very intuitive. Indeed, no animal rejects the member of the opposite sex that it is randomly paired with, and therefore the order with which the pairs are formed should not matter. This elementary result has many important consequences. First of all, since the mating pattern $Q(T)$ is a function of the terminal pair-list $\mathcal{L}(T)$, the distribution of $Q(T)$ under $\mathbb{P}^\t$ for every $\t\in\Phi$ can be easily computed and turns out to be multiple hypergeometric. Therefore, the expected mating pattern under $\mathbb{P}^\t$ is a contingency table in product form for every $\t\in\Phi$. In particular, the species is panmictic as one would predict. 

Our main result in the case of definite mating upon encounter is Theorem \ref{kungfu} which gives the distribution of $Q(t)$ under $\mathbb{P}$ for any $t\ge0$ (rather than just the terminal time $T$). This result is yet another consequence of Theorem \ref{bekirog}. Indeed, the proof of Theorem \ref{kungfu} uses the alternative representation of the model, and relies on the observation that the pair-list process $\mathcal{L}(\cdot)$ is measurable with respect to $\mathcal{L}(T)$ and $\tau = (\tau_s(\zeta))_{s\ge1}$ which are independent under $\mathbb{P}$ (as shown in Theorem \ref{bekirog}). Finally, for any $t\ge0$, once the distribution of $Q(t)$ under $\mathbb{P}$ is known, the expected pair-type matrix $\mathbb{E}[Q(t)]$ is easily computed.

We are able to obtain all of these exact formulas in the special case of definite mating upon encounter, but this is unfortunately not representative of the generic case. In Section \ref{defmatemixsection}, as a first step in dealing with nontrivial mating preferences, we consider the SEM model under the assumption that, for every $i,j\in\{1,\ldots,k\}$, $p_{ij} = 1$ holds if and only if $i \ne j$. In words, there is definite mating upon encounter for mixed-type temporary pairs only. Here and throughout, pure-type (resp.\ mixed-type) refers to animals in a pair having the same (resp.\ different) types. In this case,  the expectation of $Q_{ii}(T)$ under $\mathbb{P}$ for pure-type pairs should be strictly less than what it was in the previous case. This prediction turns out to be true even under $\mathbb{P}^\t$ for every $\t\in\Phi$, and its precise statement constitutes Theorem \ref{aaolurmu}. 

At first sight, the special cases of mating preferences considered in Sections \ref{defmatesection} and \ref{defmatemixsection} seem to be limited in their scope. However, as we will outline below, when the point processes $\{N(\zeta)\}_{\zeta\in Z}$ are Poisson or Bernoulli, it is possible to generalize our results for definite mating upon encounter to a wide class of mating preferences.

\subsubsection{Poisson firing times}

In Section \ref{poissonsection}, we consider the SEM model under the assumption that the point processes $\{N(\zeta)\}_{\zeta\in Z}$ are Poisson. For any $\zeta\in Z$, we denote the intensity of $N(\zeta)$ by $\alpha_i$ (resp.\ $\beta_j$) if $\zeta$ is a type-$i$ female (resp.\ type-$j$ male). In this case, it is clear that $Q(\cdot)$ is a continuous-time pure jump Markov chain under $\mathbb{P}$. Moreover, its entries have jumps of size exactly $1$ since there are no multiple firings at any time. In Section \ref{poiinfsec}, we give the infinitesimal generator of $Q(\cdot)$ which turns out to depend on $p_{ij}$, $\alpha_i$ and $\beta_j$ only through
\begin{equation}\label{sdemir}
\pi_{ij} := p_{ij}(\alpha_i + \beta_j).
\end{equation}
Therefore, we have the freedom to change these parameters as long as the $k\times k$ matrix $$\Pi := (\pi_{ij})$$ stays the same. This observation plays a key role in our analysis. We finish Section \ref{poiinfsec} by giving recursive equations for $u_{ij}(t;x_1,\ldots,x_k;y_1,\ldots,y_k)$ and $u^*_{ij}(x_1,\ldots,x_k;y_1,\ldots,y_k)$.

Next, in Section \ref{poifinesec}, we introduce the so-called Poisson fine balance condition on $\Pi$ which requires that
$\pi_{ij} + \pi_{i'j'} = \pi_{ij'} + \pi_{i'j}$
for every $i,j,i',j'\in\{1,\ldots,k\}$. We motivate this condition by showing that it is necessary for panmixia. Then, we prove that the Poisson fine balance condition holds if and only if there exist $\bar\alpha_1,\ldots,\bar\alpha_k,\bar\beta_1,\ldots,\bar\beta_k\ge0$ such that
\begin{equation}\label{selod}
\pi_{ij} = \bar\alpha_i + \bar\beta_j
\end{equation}
for every $i,j\in\{1,\ldots,k\}$.

Our main result in Section \ref{poissonsection} is Theorem \ref{yumosoglu} which assumes that the Poisson fine balance condition is satisfied, and gives the distribution of $Q(t)$ under $\mathbb{P}$ for every $t\ge0$. The proof of Theorem \ref{yumosoglu} uses the aforementioned change-of-parameters technique. Indeed, by comparing (\ref{sdemir}) and (\ref{selod}), we can assume without loss of generality that
$$p_{ij} = 1,\quad\alpha_i = \bar\alpha_i\quad\mbox{and}\quad\beta_j = \bar\beta_j$$
for every $i,j\in\{1,\ldots,k\}$. Then, we have definite mating upon encounter and Theorem \ref{kungfu} gives the desired result. The corollaries of Theorem \ref{kungfu} carry over, too. In particular, the Poisson fine balance condition characterizes panmixia in the Poisson case.

As we have mentioned in Section \ref{literatur}, Gimelfarb's individual EM model corresponds to having Poisson firing times with intensities $\alpha_i = 0$ and $\beta_j = 1$ for every $i,j\in\{1,\ldots,k\}$. Recall condition (\ref{iecon}) which Gimelfarb conjectures to be sufficient for (asymptotic) panmixia. Note that, if (\ref{iecon}) holds for some $\bar\alpha_1,\ldots,\bar\alpha_k,\bar\beta_1,\ldots,\bar\beta_k$, then
$$\pi_{ij} = p_{ij}(\alpha_i + \beta_j) = (\bar\alpha_i + \bar\beta_j)(0 + 1) = \bar\alpha_i + \bar\beta_j,$$
i.e., the Poisson fine balance condition is satisfied by (\ref{selod}). Therefore, the species is indeed panmictic. At this point, we would like to emphasize that we thereby not only settle Gimelfarb's panmixia conjecture, but also strengthen and generalize it in the following ways:
\begin{itemize}
\item We prove that the Poisson fine balance condition is sufficient for panmixia (and not only for asymptotic panmixia).
\item We show that, in fact, the Poisson fine balance condition characterizes panmixia, i.e., it is also necessary.
\item We allow the intensities $\alpha_1,\ldots,\alpha_k,\beta_1,\ldots,\beta_k$ of the Poisson firing times to be arbitrary.
\item When the Poisson fine balance condition is satisfied, we give an explicit formula for the distribution of $Q(t)$ under $\mathbb{P}$ for any $t\ge0$.
\end{itemize}

Finally, in Section \ref{poicharsec}, we assume that $k=2$, i.e., there are only two types of females and males. Now that panmixia is characterized by the Poisson fine balance condition, it is natural to ask if homogamy and heterogamy can be similarly characterized. We accomplish this in Theorem \ref{fullpull} which says that the species is
\begin{align}
&\mbox{heterogamous if\quad $\pi_{11} + \pi_{22} < \pi_{12} + \pi_{21}$,}\nonumber\\
&\mbox{panmictic if\qquad\ \ $\pi_{11} + \pi_{22} = \pi_{12} + \pi_{21}$,\quad and}\label{conconcon}\\
&\mbox{homogamous if\quad\ $\pi_{11} + \pi_{22} > \pi_{12} + \pi_{21}$.}\nonumber
\end{align}
The proof of Theorem \ref{fullpull} uses the change-of-parameters technique, too. Indeed, for example when $\pi_{11} + \pi_{22} < \pi_{12} + \pi_{21}$, we can assume without loss of generality that we have definite mating  upon encounter for mixed-type temporary pairs only, and the desired result follows from Theorem \ref{aaolurmu}. The significance of Theorem \ref{fullpull} lies in the fact that it provides a full characterization of panmixia, homogamy and heterogamy. Every $\Pi$ satisfies exactly one of the three conditions in (\ref{conconcon}), i.e., we have a trichotomy.

\subsubsection{Bernoulli firing times}

In Section \ref{bernoullisection}, we consider the SEM model under the assumption that $\{N(\zeta)\}_{\zeta\in Z}$ are Bernoulli point processes. This means that, for any $\zeta\in Z$, $N(\zeta)$ is a random subset of $\mathbb{N} = \{1,2,3,\ldots\}$ formed by independent Bernoulli trials assigned to the natural numbers. We denote the success probability of these Bernoulli trials by $\alpha_i$ (resp.\ $\beta_j$) if $\zeta$ is a type-$i$ female (resp.\ type-$j$ male). 

Our results in the case with Bernoulli firing times are very similar to those in the previous case with Poisson firing times. In order to emphasize this similarity, we have chosen to use the same wording (wherever possible) and parallel numbering for the theorems and displays in Sections \ref{poissonsection} and \ref{bernoullisection}. This way, these two sections can be read independently of each other, and their contents can be easily compared and contrasted.

Having said this, here are some of the important points where the Bernoulli case differs from the Poisson case.
\begin{itemize}
\item The pair-type process $Q(\cdot)$ is a discrete-time (as opposed to continuous-time) Markov chain under $\mathbb{P}$. The entries of $Q(\cdot)$ can have integer jumps of size greater than $1$ since multiple firings are possible at any time $t\in\mathbb{N}$. 
\item The transition kernel of $Q(\cdot)$ 
depends on $p_{ij}$, $\alpha_i$ and $\beta_j$ only through
$$\pi_{ij} := p_{ij}(\alpha_i + \beta_j - \alpha_i\beta_j),$$
and therefore, we have the freedom to change these parameters as long as $\Pi := (\pi_{ij})$ is the same.
\item The Poisson fine balance condition is replaced by the so-called Bernoulli fine balance condition which requires that
$$(1 - \pi_{ij})(1 - \pi_{i'j'}) = (1 - \pi_{ij'})(1 - \pi_{i'j})$$
for every $i,j,i',j'\in\{1,\ldots,k\}$.
\item The Bernoulli fine balance condition holds if and only if $\exists\ \bar\alpha_1,\ldots,\bar\alpha_k,\bar\beta_1,\ldots,\bar\beta_k\in[0,1]$ such that
$$1 - \pi_{ij} = (1 - \bar\alpha_i)(1 - \bar\beta_j)$$
for every $i,j\in\{1,\ldots,k\}$.
\end{itemize}

Our main result in Section \ref{bernoullisection} is Theorem \ref{yumosoglup} which is completely parallel to Theorem \ref{yumosoglu}. Its proof uses the change-of-parameters idea and thus reduces the model to definite mating upon encounter. The distributions that Theorems \ref{yumosoglu} and \ref{yumosoglup} provide for $Q(t)$ under $\mathbb{P}$ are almost identical, except for the natural difference due to continuous-time vs.\ discrete-time. In particular, the Bernoulli fine balance condition characterizes panmixia in the Bernoulli case.

Recall from Section \ref{literatur} that Gimelfarb's mass EM model corresponds to having Bernoulli firing times with success probabilities $\alpha_i = 0$ and $\beta_j = 1$ for every $i,j\in\{1,\ldots,k\}$. His sufficient condition (\ref{mecon}) for panmixia implies the Bernoulli fine balance condition. Indeed, if (\ref{mecon}) holds for some $\bar\gamma_1,\ldots,\bar\gamma_k,\bar\delta_1,\ldots,\bar\delta_k$, then
$$1 - \pi_{ij} = 1 - p_{ij}(\alpha_i + \beta_j - \alpha_i\beta_j) = 1 - (1 - \bar\gamma_i\bar\delta_j)(0 + 1 -0\cdot 1) =  \bar\gamma_i\bar\delta_j = (1 - \bar\alpha_i)(1 - \bar\beta_j)$$
with $\bar\alpha_i = 1 - \bar\gamma_i$ and $\bar\beta_j = 1 - \bar\delta_j$. Therefore, we not only provide a rigorous proof of Gimelfarb's panmixia result, but also  strengthen and generalize it as in the previous case.

Finally, we assume that $k=2$, and prove in Theorem \ref{fullbull} that the species is
\begin{align}
&\mbox{heterogamous if\quad $(1-\pi_{11})(1-\pi_{22}) > (1-\pi_{12})(1-\pi_{21})$,}\nonumber\\
&\mbox{panmictic if\qquad\ \ $(1-\pi_{11})(1-\pi_{22}) = (1-\pi_{12})(1-\pi_{21})$,\quad and}\nonumber\\
&\mbox{homogamous if\quad\ $(1-\pi_{11})(1-\pi_{22}) < (1-\pi_{12})(1-\pi_{21})$.}\nonumber
\end{align}
Note that this trichotomy is completely parallel to (\ref{conconcon}) from the previous case.

\section{General firing times}\label{generalsection}

\subsection{Assumptions and further notation}\label{technicalpoints}

In this section, we will make the following rather general assumptions regarding the firing times of the animals.

\begin{enumerate}
\item [(Gen1)] $\{N(\zeta)\}_{\zeta\in Z}$ are mutually independent simple point processes on $[0,\infty)$.
\item [(Gen2)]$N(\zeta)$ are identically distributed for all type-$i$ females (resp.\ type-$j$ males). 
\item [(Gen3)] $N(\zeta)\{0\}$ is almost surely zero for every $\zeta\in Z$.
\item [(Gen4)] $N(\zeta)[0,\infty)$ is almost surely either zero or infinite for each $\zeta\in Z$. Moreover, if it is zero for a female (resp.\ male), then it is infinite for all males (resp.\ females).
\end{enumerate}
Explanation: (Gen1) and (Gen2) were already mentioned in Section \ref{themodel}. (Gen3) is equivalent to saying that the firing times are almost surely positive, which implies that $\mathcal{L}(0) = \emptyset.$ Finally, on one hand, (Gen4) allows certain animals to not fire at all, on the other hand, in combination with the $p_{ij}>0$ assumption, it ensures that such animals eventually mate (upon an encounter initiated by a member of the opposite sex). Recall Definition \ref{topgaz} and note that (Gen1)--(Gen4) imply $\tau\in\Phi$ almost surely. Here and throughout, as a convention, if $N(\zeta)$ has no time points, i.e., $N(\zeta)[0,\infty) = 0$, we set $\tau_s(\zeta) = \infty$ for every $s\ge1$.

Let $\mathcal{M}^{k\times k}(\mathbb{N}\cup\{0\})$ be the set of $k\times k$ matrices with nonnegative integer entries, equipped with the following partial order: $M\le M'$ if and only if $m_{ij}\le m'_{ij}$ for every $i,j\in\{1,\ldots,k\}$. Denote the $i$th row sum, the $j$th column sum and the grand total of any $M = (m_{ij})\in \mathcal{M}^{k\times k}(\mathbb{N}\cup\{0\})$ by
$$m_{i,\cdot} = \sum_{j'=1}^km_{ij'},\qquad m_{\cdot,j} = \sum_{i'=1}^km_{i'j}\qquad\mbox{and}\qquad m_{tot} = \sum_{i'=1}^k\sum_{j'=1}^k m_{i'j'},$$
respectively. With this notation, the pair-type process $Q(\cdot)$ takes values in
\begin{equation}\label{eeeh}
\mathcal{E} = \mathcal{E}(x_1,\ldots,x_k;y_1,\ldots,y_k) = \{M\in\mathcal{M}^{k\times k}(\mathbb{N}\cup\{0\}): m_{i,\cdot} \le x_i\ \mbox{and}\ m_{\cdot,j} \le y_j\}
\end{equation}
and its initial value is $Q(0) = 0$, the zero matrix.

At any time $t\in[0,T]$, the number of single type-$i$ females, the number of single type-$j$ males and the total number of single females (or males) are equal to
$$x_i - Q_{i,\cdot}(t),\qquad y_j - Q_{\cdot,j}(t)\qquad\mbox{and}\qquad n - Q_{tot}(t),$$
respectively. In particular, the mating pattern $Q(T)$ takes values in the set
$$\mathcal{E}' = \mathcal{E}'(x_1,\ldots,x_k;y_1,\ldots,y_k) = \{M\in\mathcal{M}^{k\times k}(\mathbb{N}\cup\{0\}): m_{i,\cdot} = x_i\ \mbox{and}\ m_{\cdot,j} = y_j\}$$
of all $k\times k$ contingency tables with $i$th row sum (resp.\ $j$th column sum) equal to $x_i$ (resp.\ $y_j$). Note that $|\mathcal{E}'|>1$ if and only if $$|\{i=1,\ldots,k: x_i > 0\}|>1\qquad\mbox{and}\qquad |\{j=1,\ldots,k: y_j > 0\}|>1.$$

It is clear that the pair-type matrix $Q(t)$ at any $t\ge0$ is measurable with respect to the pair-list $\mathcal{L}(t)$. Indeed, define a function $\gamma_1: \{\varphi_1,\ldots,\varphi_n\}\to\{1,\ldots,k\}$ by setting $\gamma_1(\varphi_a) = i$ if $\varphi_a$ is of type-$i$. Similarly, define a function $\gamma_2:\{\mu_1,\ldots,\mu_n\}\to\{1,\ldots,k\}$ which encodes the type of each male. Finally, let $\Gamma(\varphi_a,\mu_b) := (\gamma_1(\varphi_a),\gamma_2(\mu_b))$. With this notation,
\begin{align}
x_i &= \sum_{a=1}^n\one_{\{\gamma_1(\varphi_a) = i\}},\qquad y_j = \sum_{b=1}^n\one_{\{\gamma_2(\mu_b) = j\}}\qquad \mbox{and}\nonumber\\
Q_{ij}(t) &= \sum_{a=1}^n\sum_{b=1}^n\one_{\{(\varphi_a,\mu_b)\in \mathcal{L}(t),\, \Gamma(\varphi_a,\mu_b) = (i,j)\}}.\label{sancak}
\end{align}

Observe that the terminal pair-list $\mathcal{L}(T)$ consists of exactly $n$ pairs, i.e., $|\mathcal{L}(T)| = n$, and that no two pairs in it have a common animal. We will refer to the latter property as being admissible. Each realization of $\mathcal{L}(T)$ can be identified with an element $\sigma$ of the symmetric group $\Sigma_n$ (i.e., a permutation of $\{1,\ldots,n\}$) in the following way:
\begin{equation}\label{identify}
\sigma(a) = b\qquad\iff\qquad(\varphi_a,\mu_b)\in \mathcal{L}(T).
\end{equation}
We will abbreviate this identification as $\mathcal{L}(T)=\mathcal{C}_{\sigma}$. With this notation, $\mathcal{L}(T)$ is sampled from $$\Lambda_n := \{\mathcal{C}_{\sigma}: \sigma\in \Sigma_n\}.$$

\subsection{An alternative representation of the SEM model}\label{alternativerepsection}

Recall (\ref{guzel}) and the sampling procedure at the encounter stage of the first firing round. The following lemma gives the probability measure this procedure induces on the set of admissible collections of pairs.

\begin{lemma}\label{disordered}
Given any $\t\in\Phi$, fix an admissible collection $\mathcal{C}$ of pairs such that
\begin{itemize}
\item [(i)] each $\zeta\in Z$ with $t_1(\zeta) = t_1^*$ is in a pair in $\mathcal{C}$, and
\item [(ii)] each pair $(\varphi_a,\mu_b)\in\mathcal{C}$ satisfies $t_1(\varphi_a)\wedge t_1(\mu_b) = t_1^*$.
\end{itemize}
Then,
\begin{equation}\label{permutation}
\mathbb{P}^\t((\Delta\mathcal{L})'(t_1^*) = \mathcal{C}) = \frac1{n(n-1)\cdots(n-|\mathcal{C}|+1)} = \frac{(n-|\mathcal{C}|)!}{n!}.
\end{equation}
\end{lemma}

\begin{proof}
We will show this by induction on $n\ge1$. If $n=1$, then there is only one possible collection $\mathcal{C} := \{(\varphi_1,\mu_1)\}$, which is consistent with $$\frac{(n-|\mathcal{C}|)!}{n!} = \frac{(1-1)!}{1!} = 1.$$
For any $n\ge2$, assume that the desired result holds for $n-1$. The probability that the first animal (with respect to the order on $\{\zeta\in Z: t_1(\zeta) = t_1^*\}$), say $\varphi_a$, indeed samples whoever it is paired up with in $\mathcal{C}$, say $\mu_b$, is equal to $\frac1{n}$. If $|\mathcal{C}|=1$, then we are done; if not, then we have the following left: (i) $n-1$ females and males for the subsequent samplings at $t_1^*$; and (ii) $|\mathcal{C}|-1$ pairs in $\mathcal{C}\setminus\{(\varphi_a,\mu_b)\}$. Therefore, the sought probability is equal to the right-hand side of (\ref{permutation}) by the induction hypothesis.
\end{proof}

\begin{remark}\label{bahaii}
It is evident from this proof that the distribution of $(\Delta\mathcal{L})'(t_1^*)$ under $\mathbb{P}^\t$ does not depend on the order imposed on $\{\zeta\in Z: t_1(\zeta) = t_1^*\}$.
\end{remark}

Observe that if we sample a $\sigma\in \Sigma_n$ uniformly at random and identify it with an admissible collection $\mathcal{C}_{\sigma}$ of $n$ pairs via (\ref{identify}), then the probability that it contains (as a subset) a given admissible collection $\mathcal{C}$ is equal to the right-hand side of (\ref{permutation}). This simple observation provides us with the following alternative representation of the SEM model:
Right before the $r$th firing round starts, the pair-list is equal to $\mathcal{L}(t_{r-1}^*)$, and there are $n - Q_{tot}(t_{r-1}^*)\ge 1$ single females and males. At $t_r^*$, relabel these animals $1,\ldots,n - Q_{tot}(t_{r-1}^*)$ and pair them up according to $\mathcal{C}_{\sigma}$ where $\sigma\in \Sigma_{n - Q_{tot}(t_{r-1}^*)}$ is sampled uniformly at random. Next, form $(\Delta\mathcal{L})'(t_r^*)$ by temporarily keeping only those pairs that have at least one animal that fires at the $r$th round. Finally, for each such type-$ij$ temporary pair, sample a Bernoulli random variable with success probability $p_{ij}$ and discard the pair in the event of failure. This gives us the set $\Delta\mathcal{L}(t_r^*)$ of new permanent pairs which are then added to $\mathcal{L}(t_{r-1}^*)$ to form the pair-list $\mathcal{L}(t_r^*) = \mathcal{L}(t_{r-1}^*) \cup \Delta\mathcal{L}(t_r^*)$.

Although the rather top-down encounter mechanism  in this representation of the SEM model is perhaps not natural from a biological point of view, it will turn out to be very convenient for mathematical analysis.

\subsection{Definite mating upon encounter}\label{defmatesection}

Assume that $p_{ij} = 1$ for every $i,j\in\{1,\ldots,k\}$. Then, the two stages of the SEM model are reduced to one since every encounter definitely results in mating. In particular, each animal mates by its first firing time at the latest, and therefore its subsequent firing times are irrelevant. In this case, it is intuitively clear that the terminal pair-list $\mathcal{L}(T)$ should be uniformly distributed on $\Lambda_n = \{\mathcal{C}_{\sigma}: \sigma\in \Sigma_n\}$ under $\mathbb{P}$. In fact, we have a much stronger result. 

\begin{theorem}\label{bekirog}
If $p_{ij}=1$ for every $i,j\in\{1,\ldots,k\}$, then the terminal pair-list $\mathcal{L}(T)$ is uniformly distributed on $\Lambda_n$ under $\mathbb{P}^\t$ for every $\t\in\Phi$. In particular, $\mathcal{L}(T)$ and $\tau = (\tau_s(\zeta))_{s\ge1,\zeta\in Z}$ are independent under $\mathbb{P}$.
\end{theorem}

\begin{proof}
We will prove this by strong induction on $n\ge1$. The case $n=1$ is trivial since there is only one possible collection $\{(\varphi_1,\mu_1)\}$ of pairs. For any $n\ge2$, assume that the desired result holds for all $1\le n'\le n-1$. Fix a $\sigma\in \Sigma_n$ and let $$\mathcal{C}_{\sigma}(t_1^*) = \{(\varphi_a,\mu_b)\in\mathcal{C}_{\sigma}:\tau_1(\varphi_a)\wedge\tau_1(\mu_b) = t_1^*\}.$$
Since $\mathcal{L}(0) = \emptyset$ and every encounter results in mating, we have $\mathcal{L}(t_1^*)  = \Delta\mathcal{L}(t_1^*) = (\Delta\mathcal{L})'(t_1^*)$. Therefore,
$$\mathbb{P}^\t(\mathcal{L}(t_1^*) = \mathcal{C}_{\sigma}(t_1^*)) = \frac{(n-|\mathcal{C}_{\sigma}(t_1^*)|)!}{n!}$$
by Lemma \ref{disordered} (or, equivalently, the alternative representation of the model). If $|\mathcal{C}_{\sigma}(t_1^*)|=n$, then $\mathcal{C}_{\sigma} = \mathcal{C}_{\sigma}(t_1^*)$, $T =  t_1^*$, and we are done; if not, we have the following left: (i) $n-|\mathcal{C}_{\sigma}(t_1^*)|$ single females and males; and (ii) $n-|\mathcal{C}_{\sigma}(t_1^*)|$ pairs in $\mathcal{C}_{\sigma}\setminus\mathcal{C}_{\sigma}(t_1^*)$. Therefore,
\begin{align*}
\mathbb{P}^\t(\mathcal{L}(T) = \mathcal{C}_{\sigma}) &= \mathbb{P}^\t(\mathcal{L}(t_1^*) = \mathcal{C}_{\sigma}(t_1^*))\mathbb{P}^\t(\mathcal{L}(T) = \mathcal{C}_{\sigma}\,|\,\mathcal{L}(t_1^*) = \mathcal{C}_{\sigma}(t_1^*))\\
&= \frac{(n-|\mathcal{C}_{\sigma}(t_1^*)|)!}{n!}\frac1{(n-|\mathcal{C}_{\sigma}(t_1^*)|)!} = \frac1{n!}
\end{align*}
by the induction hypothesis.
\end{proof}

Since the mating pattern $Q(T)$ is a function of the terminal pair-list $\mathcal{L}(T)$, we immediately obtain the following corollaries.

\begin{corollary}\label{korcount}
Take a population with $x_1,\ldots,x_k$ females and $y_1,\ldots,y_k$ males of types $1,\ldots,k$, respectively, such that (\ref{destur}) holds. If $p_{ij}=1$ for every $i,j\in\{1,\ldots,k\}$, then
\begin{equation*}
\mathbb{P}^\t(Q(T) = M) = \frac{\Big(\prod_{i} x_i!\Big)\left(\prod_{j} y_j!\right)}{n!\left(\prod_{i,j}m_{ij}!\right)}
\end{equation*}
for every $\t\in\Phi$ and $M\in\mathcal{E}'$, i.e., the distribution of $Q(T)$ under $\mathbb{P}^\t$ (and, therefore, under $\mathbb{P}$) is multiple hypergeometric.
\end{corollary}

\begin{proof}
By Theorem \ref{bekirog}, this is just a counting exercise. See \cite[p.\ 247]{Dav02} for the multiple hypergeometric distribution which appears in the tests of randomness for contingency tables. Our desired result is stated on that page, too. Its proof is easy and we leave it to the reader.
\end{proof}

\begin{corollary}\label{korpan}
If $p_{ij}=1$ for every $i,j\in\{1,\ldots,k\}$, then
$$u_{ij}^*(x_1,\ldots,x_k;y_1,\ldots,y_k\,|\,\t) = \mathbb{E}^\t[Q_{ij}(T)] = \frac{x_iy_j}{n}$$
for every $\t\in\Phi$ and every $x_1,\ldots,x_k,y_1,\ldots,y_k\ge0$ satisfying (\ref{destur}). Consequently, the species is panmictic.
\end{corollary}

\begin{proof}
By Corollary \ref{korcount}, this follows from the formula for the mean of a multiple hypergeometric distribution which is stated in \cite[p.\ 247]{Dav02} and derived exactly in the same way as that of a hypergeometric distribution.
\end{proof}

Under our current assumption of definite mating upon encounter, the pair-list process $\mathcal{L}(\cdot)$ is measurable with respect to the terminal pair-list $\mathcal{L}(T)$ and the firing times $\tau = (\tau_s(\zeta))_{s\ge1, \zeta\in Z}$. Indeed, $\mathcal{L}(T)$ is the list of ``who is destined to be with whom", and
\begin{equation}\label{kompak}
\mathcal{L}(t) = \{(\varphi_a,\mu_b)\in \mathcal{L}(T): \tau_1(\varphi_a)\wedge\tau_1(\mu_b) \le t\}
\end{equation}
for $t\ge 0$. Due to the independence of $\mathcal{L}(T)$ and $(\tau_s(\zeta))_{s\ge1, \zeta\in Z}$ (established in Theorem \ref{bekirog}), this approach is not only conceptually elegant but also computationally practical, as we will see in the theorem below.

Note that, naturally, only the first firing time of each animal appears in (\ref{kompak}). Denote the cumulative distribution function of these times by
$$F_i(t) := \nu(\tau_1(\varphi_a)\le t)\quad\mbox{if}\quad\gamma_1(\varphi_a) = i\qquad\mbox{and}\qquad G_j(t) := \nu(\tau_1(\mu_b)\le t)\quad\mbox{if}\quad\gamma_2(\mu_b) = j.$$
Since the pair-type process $Q(\cdot)$ is measurable with respect to $\mathcal{L}(\cdot)$ via (\ref{sancak}), we have the following result.

\begin{theorem}\label{kungfu}
Take a population with $x_1,\ldots,x_k$ females and $y_1,\ldots,y_k$ males of types $1,\ldots,k$, respectively, such that (\ref{destur}) holds. If $p_{ij}=1$ for every $i,j\in\{1,\ldots,k\}$, then
\begin{equation*}
\mathbb{P}\left(Q(t) = M\right) = \frac{\Big(\prod_{i} x_i!\Big)\left(\prod_{j} y_j!\right)}{n!\left(\prod_{i,j}m_{ij}!\right)}\sum_{\substack{M'\in\mathcal{E}':\\M'\ge M}}\prod_{i,j}\frac{\left(\lambda_{ij}(t)\right)^{m_{ij}}\left(1-\lambda_{ij}(t)\right)^{m'_{ij}-m_{ij}}}{(m'_{ij} - m_{ij})!}
\end{equation*}
for every $t\ge0$ and $M\in\mathcal{E}$, where 
\begin{equation}\label{gammaz}
\lambda_{ij}(t) := 1 - (1 - F_i(t))(1 - G_j(t)) = F_i(t) + G_j(t) - F_i(t)G_j(t).
\end{equation}
\end{theorem}

\begin{proof}
For every $t\ge0$ and $M\in\mathcal{E}$, we have
\begin{align}
\mathbb{P}\left(Q(t) = M\right) &= \sum_{\substack{M'\in\mathcal{E}':\\M'\ge M}}\mathbb{P}\left(Q(T) = M'\right)\mathbb{P}\left(Q(t) = M\,|\,Q(T) = M'\right)\nonumber\\
&= \sum_{\substack{M'\in\mathcal{E}':\\M'\ge M}}\mathbb{P}\left(Q(T) = M'\right)\prod_{i,j}{m'_{ij}\choose m_{ij}}\left(\lambda_{ij}(t)\right)^{m_{ij}}\left(1-\lambda_{ij}(t)\right)^{m'_{ij}-m_{ij}}\label{ramos}\\
&= \sum_{\substack{M'\in\mathcal{E}':\\M'\ge M}}\frac{\Big(\prod_{i} x_i!\Big)\left(\prod_{j} y_j!\right)}{n!\left(\prod_{i,j}m'_{ij}!\right)}\prod_{i,j}{m'_{ij}\choose m_{ij}}\left(\lambda_{ij}(t)\right)^{m_{ij}}\left(1-\lambda
_{ij}(t)\right)^{m'_{ij}-m_{ij}}\label{ahilli}\\
&= \frac{\Big(\prod_{i} x_i!\Big)\left(\prod_{j} y_j!\right)}{n!\left(\prod_{i,j}m_{ij}!\right)}\sum_{\substack{M'\in\mathcal{E}':\\M'\ge M}}\prod_{i,j}\frac{\left(\lambda_{ij}(t)\right)^{m_{ij}}\left(1-\lambda_{ij}(t)\right)^{m'_{ij}-m_{ij}}}{(m'_{ij} - m_{ij})!}.\nonumber
\end{align}
Explanation: $Q(T)$ and $(\tau_s(\zeta))_{s\ge1,\zeta\in Z}$ are independent by Theorem \ref{bekirog}. Recall (\ref{kompak}) and note that each type-$ij$ pair in $\mathcal{L}(T)$ is contained in $\mathcal{L}(t)$ with probability $\lambda_{ij}(t)$ and independently of all other pairs. Therefore, the conditional distribution of $Q_{ij}(t)$ given $Q_{ij}(T)$ is binomial with parameters $Q_{ij}(T)$ and $\lambda_{ij}(t)$. This gives (\ref{ramos}). Finally, (\ref{ahilli}) follows from Corollary \ref{korcount}.
\end{proof}

\begin{corollary}\label{aikido}
If $p_{ij}=1$ for every $i,j\in\{1,\ldots,k\}$, then
$$u_{ij}(t; x_1,\ldots,x_k;y_1,\ldots,y_k) =\mathbb{E}[Q_{ij}(t)] = \frac{x_iy_j\lambda_{ij}(t)}{n}$$
for every $t\ge0$ and $x_1,\ldots,x_k,y_1,\ldots,y_k\ge0$ satisfying (\ref{destur}), where $\lambda_{ij}(t)$ is defined in (\ref{gammaz}).
\end{corollary}

\begin{proof}
We have seen in the proof of Theorem \ref{kungfu} that $$Q_{ij}(t)\,|\,Q_{ij}(T)\sim B(Q_{ij}(T), \lambda_{ij}(t))$$ where $B$ denotes the binomial distribution. Therefore,
$$\mathbb{E}[Q_{ij}(t)] = \mathbb{E}[\mathbb{E}[Q_{ij}(t)\,|\,Q_{ij}(T)]] = \mathbb{E}[Q_{ij}(T)]\lambda_{ij}(t) = \frac{x_iy_j \lambda_{ij}(t)}{n}$$
by the law of total expectation and Corollary \ref{korpan}.
\end{proof}

\subsection{Definite mating upon encounter for mixed-type temporary pairs only}\label{defmatemixsection}

Assume that, for every $i,j\in\{1,\ldots,k\}$, $p_{ij} = 1$ holds if and only if $i\neq j$. In words, there is definite mating upon encounter for mixed-type temporary pairs only. In this case, it is intuitively clear that the expected number of pure-type pairs should be strictly less than what it was when we had definite mating upon encounter for all temporary pairs. This assertion turns out to be true not only under $\mathbb{P}$ but also under $\mathbb{P}^\t$ for every $\t\in\Phi$. We start with two simple observations.

\begin{lemma}\label{monoblokdat}
For $x'_1,\ldots,x'_k,y'_1,\ldots,y'_k\ge0$ with $n' := x'_1 + \cdots + x'_k = y'_1 + \cdots + y'_k$, define
\begin{equation}\label{kargik}
H(x'_1,\ldots,x'_k; y'_1,\ldots,y'_k) = \left\{\begin{array}{ll}\frac{x'_1y'_1}{n'} &\mbox{if }n'\ne 0,\ \mbox{and}\\ 0&\mbox{if }n' = 0.\end{array}\right.
\end{equation}
Then, on the domain $0\le c_1 \le x'_1\wedge y'_1,\ldots,0\le c_k \le x'_k\wedge y'_k$,
\begin{equation}\label{isidd}
H(x'_1 - c_1,\ldots,x'_k - c_k; y'_1 - c_1,\ldots,y'_k - c_k) + c_1
\end{equation} is increasing in $c_i$ for every $i\in\{1,\ldots,k\}$. Moreover, if $x'_1x'_2y'_1y'_2\ne 0$, then (\ref{isidd}) is strictly increasing in $c_1$ (resp.\ $c_2$) when $c_2<x'_2\wedge y'_2$ (resp.\ $c_1<x'_1\wedge y'_1$).
\end{lemma}

\begin{proof}
It is clear from (\ref{kargik}) that (\ref{isidd}) is increasing in $c_i$ for every $i\ne1$. Moreover, if $x'_1x'_2y'_1y'_2\neq0$, then  (\ref{isidd}) is strictly increasing in $c_2$ when $c_1<x'_1\wedge y'_1$ since the latter implies $(x'_1 - c_1)(y'_1 - c_1) \ne 0$.
Finally, regarding $c_1$, observe that
$$\frac{d}{dc_1}\left(\frac{(x'_1 - c_1)(y'_1 - c_1)}{n' - (c_1+\cdots+c_k)} + c_1\right) = \left(\frac{x'_1 - c_1}{n' - (c_1+\cdots+c_k)} - 1\right)\left(\frac{y'_1 - c_1}{n' - (c_1+\cdots+c_k)} - 1\right)\ge0.$$
If $x'_1x'_2y'_1y'_2\neq0$, then this inequality is strict when $c_2<x'_2\wedge y'_2$. This concludes the proof.
\end{proof}

\begin{lemma}\label{ddeler}
For every mating preference matrix $P$ and every $\t\in\Phi$, we have
$$\mathbb{P}^\t(Q(t_1^*) \neq 0) \ge \min_{1\le i,j\le k}p_{ij} > 0.$$
\end{lemma}

\begin{proof}
Consider the first animal that fires at the first round. If it forms a permanent pair at that time, then $Q(t_1^*) \neq 0$. The former event occurs with probability at least ${\displaystyle \min_{1\le i,j\le k}p_{ij}}$.
\end{proof}

\begin{theorem}\label{aaolurmu}
Assume that, for every $i,j\in\{1,\ldots,k\}$, $p_{ij} = 1$ holds if and only if $i\neq j$. Recall the definitions \eqref{bayankut} and \eqref{nboz}. Then, we have the following results:
\begin{itemize}
\item [(a)] For every $i\in\{1,\ldots,k\}$, $x_1,\ldots,x_k,y_1,\ldots,y_k\ge0$ satisfying (\ref{destur}), and $\t\in\Phi$,
$$u_{ii}^*(x_1,\ldots,x_k;y_1,\ldots,y_k\,|\,\t) \le \frac{x_iy_i}{n}.$$
In particular,
$$u_{ii}^*(x_1,\ldots,x_k;y_1,\ldots,y_k) \le \frac{x_iy_i}{n}.$$
\item [(b)] The inequalities in part (a) are strict whenever $x_ix_{i'}y_iy_{i'}\neq 0$ for some $i'\neq i$.
\end{itemize}
\end{theorem}

\begin{proof}[Proof]
Since we can relabel the types, it suffices to prove the desired results for $u_{11}^*$. We start with part (a) and proceed by strong induction on $n\ge1$. The case $n=1$ is trivial. Indeed, 
$$u_{11}^*(x_1,\ldots,x_k;y_1,\ldots,y_k\,|\,\t) = \left\{\begin{array}{ll}0 & \mbox{if $x_1y_1 = 0$, and}\\1 & \mbox{if $x_1y_1 = 1$.}\end{array}\right.$$
For $n\ge2$, let
\begin{equation}\label{isimmisim}
\bar u_{11}^*(x_1,\ldots,x_k;y_1,\ldots,y_k) := \sup_{\t\in\Phi} u_{11}^*(x_1,\ldots,x_k;y_1,\ldots,y_k\,|\,\t).
\end{equation}
With this notation, for every $\mathbf{t}\in\Phi$,
\begin{align}
&u^*_{11}(x_1,\ldots,x_k;y_1,\ldots,y_k\,|\,\t)\nonumber\\
&\ \le \!\!\!\!\!\!\!\!\!\!\sum_{M\in\mathcal{E}\setminus(\mathcal{E}'\cup\{0\})}\!\!\!\!\!\!\!\!\!\!\!\mathbb{P}^\t(Q(t_1^*) = M)\left[\frac{(x_1 - m_{1,\cdot})(y_1 - m_{\cdot,1})}{n - m_{tot}} + m_{11}\right]\label{hipottdat}\\
&\ \quad + \sum_{M\in\mathcal{E}'}\mathbb{P}^\t(Q(t_1^*) = M)m_{11} + \mathbb{P}^\t(Q(t_1^*) = 0)\bar u^*_{11}(x_1,\ldots,x_k;y_1,\ldots,y_k)\nonumber\\
&\ = \sum_{\substack{M\in\mathcal{E}:\\M\neq 0}}\!\mathbb{P}^\t(Q(t_1^*) = M)[H(x_1 - m_{1,\cdot},\ldots,x_k - m_{k,\cdot};y_1 - m_{\cdot,1},\ldots,y_k - m_{\cdot,k}) + m_{11}]\nonumber\\
&\ \quad +\mathbb{P}^\t(Q(t_1^*) = 0)\bar u^*_{11}(x_1,\ldots,x_k;y_1,\ldots,y_k)\nonumber\\
&\ = \sum_{M\in\mathcal{E}}\mathbb{P}^\t(Q(t_1^*) = M)[H(x_1 - m_{1,\cdot},\ldots,x_k - m_{k,\cdot};y_1 - m_{\cdot,1},\ldots,y_k - m_{\cdot,k}) + m_{11}]\nonumber\\
&\ \quad - \mathbb{P}^\t(Q(t_1^*) = 0)\frac{x_1y_1}{n} + \mathbb{P}^\t(Q(t_1^*) = 0)\bar u^*_{11}(x_1,\ldots,x_k;y_1,\ldots,y_k)\nonumber\\
&\ = \sum_{M\in\mathcal{E}}\sum_{M'}\mathbb{P}^\t(Q(t_1^*) = M,Q'(t_1^*) = M')\label{gugugdat}\\
&\quad\qquad\qquad\times[H(x_1 - m_{1,\cdot},\ldots,x_k - m_{k,\cdot};y_1 - m_{\cdot,1},\ldots,y_k - m_{\cdot,k}) + m_{11}]\nonumber\\
&\ \quad - \mathbb{P}^\t(Q(t_1^*) = 0)\frac{x_1y_1}{n} + \mathbb{P}^\t(Q(t_1^*) = 0)\bar u^*_{11}(x_1,\ldots,x_k;y_1,\ldots,y_k)\nonumber\\
&\ \le \sum_{M\in\mathcal{E}}\sum_{M'}\mathbb{P}^\t(Q(t_1^*) = M, Q'(t_1^*) = M')\label{monotondat}\\
&\quad\qquad\qquad\times[H(x_1 - m'_{1,\cdot},\ldots,x_k - m'_{k,\cdot};y_1 - m'_{\cdot,1},\ldots,y_k - m'_{\cdot,k}) + m'_{11}]\nonumber\\
&\ \quad - \mathbb{P}^\t(Q(t_1^*) = 0)\frac{x_1y_1}{n} + \mathbb{P}^\t(Q(t_1^*) = 0)\bar u^*_{11}(x_1,\ldots,x_k;y_1,\ldots,y_k)\nonumber\\
&\ = \sum_{M'\in\mathcal{E}}\!\mathbb{P}^\t(Q'(t_1^*) = M')[H(x_1 - m'_{1,\cdot},\ldots,x_k - m'_{k,\cdot};y_1 - m'_{\cdot,1},\ldots,y_k - m'_{\cdot,k}) + m'_{11}]\label{metalldat}\\
&\ \quad - \mathbb{P}^\t(Q(t_1^*) = 0)\frac{x_1y_1}{n} + \mathbb{P}^\t(Q(t_1^*) = 0)\bar u^*_{11}(x_1,\ldots,x_k;y_1,\ldots,y_k)\nonumber\\
&\ = \frac{x_1y_1}{n}  - \mathbb{P}^\t(Q(t_1^*) = 0)\frac{x_1y_1}{n} + \mathbb{P}^\t(Q(t_1^*) = 0)\bar u^*_{11}(x_1,\ldots,x_k;y_1,\ldots,y_k)\label{biliyozdat}\\
&\ = \mathbb{P}^\t(Q(t_1^*) \neq 0)\frac{x_1y_1}{n} + \mathbb{P}^\t(Q(t_1^*) = 0)\bar u^*_{11}(x_1,\ldots,x_k;y_1,\ldots,y_k).\label{buyukb}
\end{align}
The inequality in (\ref{hipottdat}) follows from the induction hypothesis since $n-m_{tot} \le n-1$ for $M\ne 0$. (Technically, one should first condition on $\mathcal{L}(t^*_1)$, and then use the induction hypothesis with the shifted firing times $(t_s(\zeta) - t^*_1)_{s\ge1,\zeta\in S(t^*_1)}$ of the single animals.) $Q'(t_1^*)$ is the $k\times k$ matrix with $ij$th entry
$$Q'_{ij}(t_1^*) := \sum_{a=1}^n\sum_{b=1}^n\one_{\{(\varphi_a,\mu_b)\in (\Delta\mathcal{L})'(t_1^*),\, \Gamma(\varphi_a,\mu_b) = (i,j)\}}$$
which is the number of type-$ij$ temporary pairs formed at the first firing round. Since there is definite mating upon encounter for mixed-type temporary pairs, the second sum in (\ref{gugugdat}) is over all $M'\in\mathcal{E}$ such that $M'\ge M$ and $m'_{ij} = m_{ij}$ for every $i\neq j$. Lemma \ref{monoblokdat} with
$$x'_i = x_i - m_{i,\cdot},\quad y'_j = y_j - m_{\cdot,j}\quad\mbox{and}\quad c_i = m'_{ii} - m_{ii}$$
gives (\ref{monotondat}). Note that $M$ does not appear inside the square brackets in (\ref{monotondat}), so we can change the order of summation and obtain (\ref{metalldat}). Finally, the equality in (\ref{biliyozdat}) follows from Corollary \ref{korpan}. Indeed, in the case of definite mating upon encounter for all temporary pairs, we have $Q'(t^*_1) = Q(t^*_1)$ and
\begin{align*}
&u^*_{11}(x_1 - m'_{1,\cdot},\ldots,x_k - m'_{k,\cdot};y_1 - m'_{\cdot,1},\ldots,y_k - m'_{\cdot,k}\,|\,\t)\\
&\quad = H(x_1 - m'_{1,\cdot},\ldots,x_k - m'_{k,\cdot};y_1 - m'_{\cdot,1},\ldots,y_k - m'_{\cdot,k})
\end{align*}
for every $M'\in\mathcal{E}$ and $\mathbf{t}\in\Phi$. The decomposition of  $u^*_{11}(x_1,\ldots,x_k;y_1,\ldots,y_k\,|\,\t) = \frac{x_1y_1}{n}$ via conditioning on $Q(t^*_1)$ gives the desired equality.

For every $\epsilon>0$, there exists a $\t'\in\Phi$ such that
$$\bar u_{11}^*(x_1,\ldots,x_k;y_1,\ldots,y_k) \le u^*_{11}(x_1,\ldots,x_k;y_1,\ldots,y_k\,|\,\t') + \epsilon,$$
and (\ref{buyukb}) gives
$$\mathbb{P}^{\t'}(Q(t_1^*) \ne 0)\bar u_{11}^*(x_1,\ldots,x_k;y_1,\ldots,y_k) \le \mathbb{P}^{\t'}(Q(t_1^*) \neq 0)\frac{x_1y_1}{n} + \epsilon.$$ 
By Lemma \ref{ddeler}, we get
$$\bar u_{11}^*(x_1,\ldots,x_k;y_1,\ldots,y_k) \le \frac{x_1y_1}{n} + \frac{\epsilon}{\min\{p_{11},\ldots,p_{kk}\}}.$$
Since $\epsilon>0$ is arbitrary, we deduce that
\begin{equation}\label{timof}
\bar u_{11}^*(x_1,\ldots,x_k;y_1,\ldots,y_k) \le \frac{x_1y_1}{n}.
\end{equation}
This concludes the proof of part (a).

We prove part (b) by strong induction, too. The case $n=1$ is vacuous since $x_1x_iy_1y_i = 0$ for every $i\ne 1$. For $n\ge2$, note that if the supremum in (\ref{isimmisim}) is not attained, then (\ref{timof}) implies the desired strict inequality. If the supremum in (\ref{isimmisim}) is attained at some $\t'\in\Phi$, then we will assume without loss of generality that $x_1x_2y_1y_2 \ne 0$ and consider two subcases:
\begin{itemize}

\item [(i)] If no type-$1$ or type-$2$ animal fires at the first round, then there exists an $M\in\mathcal{E}$ such that $M \ne 0$,
\begin{equation}\label{lazimm}
(x_1 - m_{1,\cdot})(x_2 - m_{2,\cdot})(y_1 - m_{\cdot,1})(y_2 - m_{\cdot,2}) \ne 0,
\end{equation}
and $\mathbb{P}^{\t'}(Q(t_1^*) = M) > 0$. (Indeed, in the alternative representation of the SEM model, with positive probability, at least one type-$1$ (resp.\ type-$2$) female is paired up with a type-$1$ (resp.\ type-$2$) male. Since none of these animals fire at the first round, these pairs are discarded at the encounter stage.) Therefore, the inequality in (\ref{hipottdat}) is strict by the induction hypothesis.

\item [(ii)] If at least one type-$1$ or type-$2$ animal fires at the first round, then there exist $M,M'\in\mathcal{E}$ such that $m_{11} = m_{22} = 0$, $m'_{11} \vee m'_{22} \ne 0$, (\ref{lazimm}) holds, and $$\mathbb{P}^{\t'}(Q(t_1^*) = M,Q'(t_1^*) = M')>0.$$ (Indeed, with positive probability, at least one type-$1$ (resp.\ type-$2$) female is paired up with a type-$1$ (resp.\ type-$2$) male and each of these pairs is discarded at the encounter stage or the mating stage.) Therefore, the inequality in (\ref{monotondat}) is strict by Lemma \ref{monoblokdat}. 

\end{itemize}
In both of these subcases, we deduce that (\ref{timof}) holds with strict inequality. This concludes the proof of part (b).
\end{proof}

\begin{remark}
The key step in the proof of Theorem \ref{aaolurmu} is the inequality in (\ref{monotondat}) which follows from Lemma \ref{monoblokdat}. This lemma says that, in the case of definite mating upon encounter, the expected number of type-$11$ pairs in the terminal pair-list would decrease if we were to interfere with the process at any time and discard some pure-type pairs. This is precisely how we are viewing what happens (with positive probability) at the first firing round.
\end{remark}

\begin{remark}
It is perhaps natural to look for an easier proof of Theorem \ref{aaolurmu} via a coupling with definite mating upon encounter. However, there are some obstacles to such approaches. For example, consider the following setup: $k=2$, $n=3$, $x_1 = y_1 = 2$, $x_2 = y_2 = 1$, $\gamma_1(\varphi_1) = \gamma_1(\varphi_2) = \gamma_2(\mu_1) = \gamma_2(\mu_2) = 1$, and $\gamma_1(\varphi_3) = \gamma_2(\mu_3) = 2$. Suppose the firing times are ordered as $t_1(\varphi_1) < t_2(\varphi_1) < t_1(\mu_1) < t_1(\mu_2) < \cdots$ and $\varphi_1$ samples $\mu_1$ at $t_1^* = t_1(\varphi_1)$. In the case of definite mating upon encounter, if $\mu_2$ samples $\varphi_3$ at $t_2^* = t_1(\mu_2)$, then $Q_{11}(T) = 1$. However, in the case of definite mating upon encounter for mixed-type temporary pairs only, if (i) the pair $(\varphi_1,\mu_1)$ is discarded, (ii) $\varphi_1$ samples $\mu_2$ at $t_2^* = t_2(\varphi_1)$ and (iii) $\mu_1$ samples $\varphi_2$ at $t_3^* = t_1(\mu_1)$, then $Q_{11}(T) = 2 > 1$, i.e., the inequality is in the reverse direction.
\end{remark}

\begin{remark}
Theorem \ref{aaolurmu} can be slightly generalized. For example, if we assume that $k\ge3$, $p_{11}<1$, $p_{22}\le1$, $p_{ii}=1$ for every $i\in\{3,\ldots,k\}$, and $p_{ij}=1$ for every $i,j\in\{1,\ldots,k\}$ such that $i \ne j$, then
\begin{equation}\label{dbekir}
u_{11}^*(x_1,\ldots,x_k;y_1,\ldots,y_k\,|\,\t) \le \frac{x_1y_1}{n}
\end{equation}
still holds, and this inequality is strict when $p_{22}<1$ and $x_1x_2y_1y_2 \neq 0$. However, the inequality in (\ref{dbekir}) is not necessarily strict when $p_{22}=1$ and $x_1x_2y_1y_2 \neq 0$. As a counterexample, assume that all males fire at the first round. Since $p_{ij}<1$ if and only if $i=j=1$, there are only type-$1$ females and males in the singles' pool $S(t_1^*)$, and these animals have no option other than eventually forming type-$11$ permanent pairs. Therefore, the distribution of $Q(T)$ under $\mathbb{P}^\t$ is identical to what it was in the case of definite mating upon encounter for all temporary pairs. In particular, we have equality in (\ref{dbekir}).
\end{remark}

\section{Poisson firing times}\label{poissonsection}

\subsection{The infinitesimal generator}\label{poiinfsec}

In this section, we will make the following assumptions regarding the firing times of the animals.
\begin{enumerate}
\item [(Poi1)] $\{N(\zeta)\}_{\zeta\in Z}$ are mutually independent Poisson processes on $[0,\infty)$.
\item [(Poi2)] The intensity of $N(\zeta)$ is equal to some $\alpha_i\ge0$ (resp.\ $\beta_j\ge0$) for all type-$i$ females (resp.\ type-$j$ males).
\item [(Poi3)] $\alpha_i + \beta_j >0$ for every $i,j\in\{1,\ldots,k\}$.
\end{enumerate}
Note that (Poi1)--(Poi3) are stronger than (Gen1)--(Gen4) in Section \ref{generalsection}. Indeed, (Gen3) is automatically satisfied for Poisson processes, and (Poi3) implies (Gen4).

In this case, since $\{N(\zeta)\}_{\zeta\in Z}$ are memoryless, it is clear that the pair-type process $Q(\cdot)$ is a continuous-time pure jump Markov chain under $\mathbb{P}$, with state space $\mathcal{E}$ defined in (\ref{eeeh}). Moreover, the jumps of $Q_{ij}(\cdot)$ are of size $1$ because almost surely one animal fires at each firing round. Therefore, $Q(\cdot)$ may be regarded as a multidimensional pure birth process, see \cite{Mod62}.

\begin{proposition}\label{trafogg}
Assume (Poi1)--(Poi3). Take a population with $x_1,\ldots,x_k$ females and $y_1,\ldots,y_k$ males of types $1,\ldots,k$, respectively, such that (\ref{destur}) holds. The infinitesimal generator of the continuous-time Markov chain $Q(\cdot)$ has the following formula: for every $M,M'\in\mathcal{E}$,
\begin{align}
\rho(M,M') &:= \lim_{\Delta t\to 0^+}\frac1{\Delta t}\mathbb{P}\left(\left.Q(t+\Delta t) = M'\right|Q(t) = M\right)\nonumber\\
&\; = \left\{\begin{array}{ll}\frac{\pi_{ij}\left(x_i - m_{i,\cdot}\right)\left(y_j  - m_{\cdot,j}\right)}{n - m_{tot}}&\mbox{if}\ M' = M+I^{ij},\ \mbox{and}\\ 0 & \mbox{otherwise}.\end{array}\right.\label{taylanc}
\end{align}
Here, $I^{ij}\in\mathcal{M}^{k\times k}(\mathbb{N}\cup\{0\})$ denotes the $k\times k$ matrix whose entries are zero except its $ij$th entry which is $1$, and
\begin{equation}\label{pipoi}
\pi_{ij} := p_{ij}(\alpha_i + \beta_j).
\end{equation}
\end{proposition}

\begin{proof}
Conditioned on $Q(t) = M$, there are $x_i - m_{i,\cdot}$ single type-$i$ females at time $t$ and each of them fires independently with probability $\alpha_i\Delta t + o(\Delta t)$ in $[t,t + \Delta t]$. Conditioned on the latter event, each single type-$i$ female samples a single type-$j$ male with probability $\frac{y_j  - m_{\cdot,j}}{n - m_{tot}}$. This is one of the two ways of forming a temporary type-$ij$ pair. The other way is obtained by switching the roles of females and males. Combining these two ways, we see that the conditional probability of forming a temporary type-$ij$ pair in $[t,t+\Delta t]$ given $Q(t) = M$ is
$$\frac{(\alpha_i\Delta t + \beta_j\Delta t + o(\Delta t))\left(x_i - m_{i,\cdot}\right)\left(y_j  - m_{\cdot,j}\right)}{n - m_{tot}}.$$
Finally, each temporary type-$ij$ pair becomes a permanent pair with probability $p_{ij}$, and this implies (\ref{taylanc}).
\end{proof}

\begin{remark}\label{degistira}
It follows readily from (\ref{taylanc}) that the distribution of the pair-type process $Q(\cdot)$ depends on $P = (p_{ij})$, $\alpha_1,\ldots,\alpha_k$ and $\beta_1,\ldots,\beta_k$ only through the $k\times k$ matrix $$\Pi := (\pi_{ij})$$ whose entries are defined in (\ref{pipoi}). In other words, we have the freedom to change these parameters as long as $\Pi$ stays the same. Therefore, under the assumption of Poisson firing times, the EM law of the species can be identified with $\Pi$.
\end{remark}

Since the pair-type process $Q(\cdot)$ is a time-homogeneous Markov chain with infinitesimal generator $\rho$ as in (\ref{taylanc}), the Kolmogorov backward equation for 
(\ref{tozert}) is
\begin{align}
&\frac{\partial}{\partial t}u_{ij}(t;x_1,\ldots,x_k;y_1,\ldots,y_k)\nonumber\\
&\quad = \sum_{i'=1}^k\sum_{j'=1}^k\frac{\pi_{i'j'}x_{i'}y_{j'}}{n}\left[u_{ij}(t;x_1-\delta_{1i'},\ldots,x_k-\delta_{ki'};y_1-\delta_{1j'},\ldots,y_k-\delta_{kj'}) + \delta_{ii'}\delta_{jj'}\right]\label{yonetm}\\
&\quad\qquad - \frac{z}{n}u_{ij}(t;x_1,\ldots,x_k;y_1,\ldots,y_k).\nonumber
\end{align}
Here, $\delta_{\cdot\cdot}$ denotes the Kronecker delta function, and $$z := \sum_{i=1}^k\sum_{j=1}^k\pi_{ij}x_{i}y_{j}.$$
Note that
(\ref{yonetm}) is a closed system of recursive ordinary differential equations.

Similarly, using the strong Markov property, we can decompose 
(\ref{bayankut}) with respect to the possible values of $Q(t_1^*(\tau))$ and thereby obtain the following recursion:
\begin{align}
&u^*_{ij}(x_1,\ldots,x_k;y_1,\ldots,y_k)\nonumber\\
&\quad = \sum_{i'=1}^k\sum_{j'=1}^k\frac{\pi_{i'j'}x_{i'}y_{j'}}{z}\left[u^*_{ij}(x_1 - \delta_{1i'},\ldots,x_k - \delta_{ki'};y_1 - \delta_{1j'},\ldots,y_k - \delta_{kj'}) + \delta_{ii'}\delta_{jj'}\right].\label{karaosman}
\end{align}
This was previously given by Mosteller \cite{Mos68}.

The recursive equations (\ref{yonetm}) and  (\ref{karaosman}) can be solved numerically, especially when $n$ is not too large. However, it is not clear if they can be solved analytically. Next, we will provide explicit solutions to these equations when the EM law $\Pi$ satisfies the so-called Poisson fine balance condition.

\subsection{Poisson fine balance and panmixia}\label{poifinesec}

\begin{definition}\label{poifinedef}
The EM law $\Pi$ is said to satisfy the Poisson fine balance condition if
\begin{equation}\label{fine}
\pi_{ij} + \pi_{i'j'} = \pi_{ij'} + \pi_{i'j}
\end{equation}
for every $i,j,i',j'\in\{1,\ldots,k\}$.
\end{definition}

The motivation behind this definition is the following observation.

\begin{proposition}\label{mulug}
Assume (Poi1)--(Poi3). If the species is panmictic, then the EM law $\Pi$ satisfies the Poisson fine balance condition.
\end{proposition}

\begin{proof}
Assume that the species is panmictic. Observe that (\ref{fine}) is trivial if $i=i'$ or $j=j'$. Take a population with $n=2$ and $x_1 = x_2 = y_1 = y_2 = 1$. Use recursion (\ref{karaosman}) to write
\begin{align*}
\frac1{2} &= \frac{x_1y_1}{n} = u^*_{11}(x_1,\ldots,x_k;y_1,\ldots,y_k)\\
&= \sum_{i=1}^k\sum_{j=1}^k\frac{\pi_{ij}x_{i}y_{j}}{z}\left[u^*_{11}(x_1 - \delta_{1i},\ldots,x_k - \delta_{ki};y_1 - \delta_{1j},\ldots,y_k - \delta_{kj}) + \delta_{1i}\delta_{1j}\right]\\
&= \frac1{z}\left(\pi_{11}[0+1] + \pi_{12}[0+0] + \pi_{21}[0+0] + \pi_{22}[1+0]\right)\\
&= \frac{\pi_{11} + \pi_{22}}{\pi_{11} + \pi_{12} + \pi_{21} + \pi_{22}}.
\end{align*}
Therefore, $\pi_{11} + \pi_{22} = \pi_{12} + \pi_{21}$. This is (\ref{fine}) with $i=j=1$ and $i'=j'=2$. By relabeling the types, the same argument works for any $i,j,i',j'\in\{1,\ldots,k\}$ such that $i\neq i'$ and $j\neq j'$, and we are done.
\end{proof}

It is natural to ask if the converse is true, i.e., whether the species is panmictic whenever the EM law satisfies the Poisson fine balance condition. This turns out to be true and can be verified by induction and (\ref{karaosman}). However, we will prove it by taking a much more conceptual approach. First, we need a lemma.

\begin{lemma}\label{propjoe}
The EM law $\Pi$ satisfies the Poisson fine balance condition if and only if there exist $\bar\alpha_1,\ldots,\bar\alpha_k,\bar\beta_1,\ldots,\bar\beta_k\ge0$ such that
\begin{equation}\label{thanks}
\pi_{ij} = \bar\alpha_i + \bar\beta_j
\end{equation}
for every $i,j\in\{1,\ldots,k\}$.
\end{lemma}

\begin{proof}
Assume that (\ref{thanks}) indeed holds with some $\bar\alpha_1,\ldots,\bar\alpha_k,\bar\beta_1,\ldots,\bar\beta_k\ge0$. Then,
$$\pi_{ij} + \pi_{i'j'} = (\bar\alpha_i + \bar\beta_j) + (\bar\alpha_{i'} + \bar\beta_{j'}) =  (\bar\alpha_i + \bar\beta_{j'}) + (\bar\alpha_{i'} + \bar\beta_{j}) = \pi_{ij'} + \pi_{i'j}$$
for every  $i,j,i',j'\in\{1,\ldots,k\}$, i.e., the Poisson fine balance condition is satisfied.

Conversely, if the Poisson fine balance condition holds, then assume without loss of generality (i.e., up to relabeling the male types) that $\pi_{11} = \min\{\pi_{11},\ldots,\pi_{1k}\}$. Define $$\bar\alpha_i := \pi_{i1}>0\quad\mbox{and}\quad \bar\beta_j := \pi_{1j} - \pi_{11}\ge 0.$$
With this notation, we have
$$\pi_{ij} = \pi_{i1} + \pi_{1j} - \pi_{11} = \bar\alpha_i + \bar\beta_j,$$
where the first equality follows from (\ref{fine}). This concludes the proof.
\end{proof}

\begin{theorem}\label{yumosoglu}
Assume (Poi1)--(Poi3). Take a population with $x_1,\ldots,x_k$ females and $y_1,\ldots,y_k$ males of types $1,\ldots,k$, respectively, such that (\ref{destur}) holds. If the EM law $\Pi$ satisfies the Poisson fine balance condition, then
\begin{equation}\label{cikolata1}
\mathbb{P}\left(Q(t) = M\right) = \frac{\Big(\prod_{i} x_i!\Big)\left(\prod_{j} y_j!\right)}{n!\left(\prod_{i,j}m_{ij}!\right)}\sum_{\substack{M'\in\mathcal{E}':\\M'\ge M}}\prod_{i,j}\frac{\left(1-e^{-\pi_{ij}t}\right)^{m_{ij}}\left(e^{-\pi_{ij}t}\right)^{m'_{ij}-m_{ij}}}{(m'_{ij} - m_{ij})!}
\end{equation}
for every $t\ge0$ and $M\in\mathcal{E}$. Similarly,
\begin{align}
\mathbb{P}\left(Q(T) = M\right) &= \frac{\Big(\prod_{i} x_i!\Big)\left(\prod_{j} y_j!\right)}{n!\left(\prod_{i,j}m_{ij}!\right)}\quad\text{for $M\in\mathcal{E}'$},\label{cikolata3}\\
u^*_{ij}(x_1,\ldots,x_k;y_1,\ldots,y_k) &= \frac{x_iy_j}{n},\quad\mbox{and}\label{cikolata4}\\
u_{ij}(t;x_1,\ldots,x_k;y_1,\ldots,y_k) &= \frac{x_iy_j}{n}\left(1-e^{-\pi_{ij}t}\right)\label{cikolata2}.
\end{align}
\end{theorem}

\begin{proof}
If the EM law satisfies the Poisson fine balance condition, then Lemma \ref{propjoe} implies that (\ref{thanks}) holds with some $\bar\alpha_1,\ldots,\bar\alpha_k,\bar\beta_1,\ldots,\bar\beta_k\ge0$. In the light of Remark \ref{degistira}, as far as the distribution of the pair-type process $Q(\cdot)$ is concerned, we can change (if necessary) $P = (p_{ij})$, $\alpha_1,\ldots,\alpha_k$ and $\beta_1,\ldots,\beta_k$, and
assume that $$p_{ij} = 1,\quad\alpha_i = \bar\alpha_i\quad\mbox{and}\quad\beta_j = \bar\beta_j$$
for every $i,j\in\{1,\ldots,k\}$. But then, we have definite mating upon encounter and Theorem \ref{kungfu} is applicable. Since $\tau_1(\zeta)$ is exponentially distributed with rate $\bar\alpha_i$ (resp.\ $\bar\beta_j$) if $\zeta$ is a type-$i$ female (resp.\ type-$j$ male), we have
$$\lambda_{ij}(t) = 1 - (1 - F_i(t))(1 - G_j(t)) = 1 - e^{-\bar\alpha_it}e^{-\bar\beta_jt} = 1 - e^{-\pi_{ij}t}$$
and this implies (\ref{cikolata1}). Similarly, Corollaries \ref{korcount}, \ref{korpan} and \ref{aikido} give (\ref{cikolata3}), (\ref{cikolata4}) and (\ref{cikolata2}), respectively.
\end{proof}

\begin{corollary}\label{karaks}
Assume (Poi1)--(Poi3). The species is panmictic if and only if the EM law $\Pi$ satisfies the Poisson fine balance condition.
\end{corollary}

\begin{proof}
This is immediate from Proposition \ref{mulug} and Theorem \ref{yumosoglu}.
\end{proof}

\subsection{Full characterization of panmixia, homogamy and heterogamy in the $2\times 2$ case}\label{poicharsec}

\begin{lemma}\label{propjeff}
For $k=2$, the EM law $\Pi$ satisfies $\pi_{11} + \pi_{22} < \pi_{12} + \pi_{21}$ if and only if there exist $\bar\alpha_1,\bar\alpha_2,\bar\beta_1,\bar\beta_2\ge0$ such that
\begin{equation}\label{thankff}
\pi_{11} < \bar\alpha_1 + \bar\beta_1,\quad \pi_{12} = \bar\alpha_1 + \bar\beta_2,\quad \pi_{21} = \bar\alpha_2 + \bar\beta_1\quad\mbox{and}\quad \pi_{22} < \bar\alpha_2 + \bar\beta_2.
\end{equation}
\end{lemma}

\begin{proof}
Assume that (\ref{thankff}) indeed holds with some $\bar\alpha_1,\bar\alpha_2,\bar\beta_1,\bar\beta_2\ge0$. Then,
$$\pi_{11} + \pi_{22} < (\bar\alpha_1 + \bar\beta_1) + (\bar\alpha_2 + \bar\beta_2) =  (\bar\alpha_1 + \bar\beta_2) + (\bar\alpha_2 + \bar\beta_1) = \pi_{12} + \pi_{21}.$$

Conversely, if $\pi_{11} + \pi_{22} < \pi_{12} + \pi_{21}$, then pick $\bar\Pi = (\bar\pi_{ij})$ such that
$$\pi_{11} < \bar\pi_{11},\quad\pi_{12} = \bar\pi_{12},\quad\pi_{21} = \bar\pi_{21},\quad\pi_{22} < \bar\pi_{22}\quad\mbox{and}\quad\bar\pi_{11} + \bar\pi_{22} = \bar\pi_{12} + \bar\pi_{21}.$$
The desired result follows from Lemma \ref{propjoe} applied to $\bar\Pi$.
\end{proof}

\begin{theorem}\label{fullpull}
Assume (Poi1)--(Poi3). For $k=2$, the species is
\begin{itemize}
\item [(i)] heterogamous if\quad $\pi_{11} + \pi_{22} < \pi_{12} + \pi_{21}$,
\item [(ii)] panmictic if\qquad\ \;$\pi_{11} + \pi_{22} = \pi_{12} + \pi_{21}$, and
\item [(iii)] homogamous if\quad\ $\pi_{11} + \pi_{22} > \pi_{12} + \pi_{21}$.
\end{itemize}
\end{theorem}

\begin{proof}
Let us prove part (i). If $\pi_{11} + \pi_{22} < \pi_{12} + \pi_{21}$, then Lemma \ref{propjeff} implies that (\ref{thankff}) holds with some $\bar\alpha_1,\bar\alpha_2,\bar\beta_1,\bar\beta_2\ge0$. In the light of Remark \ref{degistira}, as far as the distribution of the pair-type process $Q(\cdot)$ is concerned, we can change (if necessary) $P = (p_{ij})$, $\alpha_1,\alpha_2$ and $\beta_1,\beta_2$, and
assume that $$p_{11} < 1,\quad p_{12} = 1,\quad p_{21} = 1,\quad p_{22} < 1,\quad\alpha_i = \bar\alpha_i\quad\mbox{and}\quad\beta_j = \bar\beta_j$$
for every $i,j\in\{1,2\}$. But then, we have definite mating upon encounter for mixed-type temporary pairs only, and Theorem \ref{aaolurmu} is applicable. This concludes the proof of part (i). We have already shown part (ii) in Corollary \ref{karaks}. Finally, part (iii) follows from part (i) by relabeling the male types.
\end{proof}

\section{Bernoulli firing times}\label{bernoullisection}

\subsection{The transition kernel}

In this section, we will make the following assumptions.
\begin{enumerate}
\item [(Ber1)] $\{N(\zeta)\}_{\zeta\in Z}$ are mutually independent Bernoulli processes on $\mathbb{N} = \{1,2,3,\ldots\}$.
\item [(Ber2)] The success probability of the Bernoulli trials of $N(\zeta)$ is equal to some $\alpha_i\in[0,1]$ (resp.\ $\beta_j\in[0,1]$) for all type-$i$ females (resp.\ type-$j$ males).
\item [(Ber3)] $\alpha_i + \beta_j >0$ for every $i,j\in\{1,\ldots,k\}$.
\end{enumerate}
Note that (Ber1)--(Ber3) are stronger than (Gen1)--(Gen4) in Section \ref{generalsection}. Indeed, (Gen3) is automatically satisfied as the Bernoulli processes are defined on $\mathbb{N}$, and (Ber3) implies (Gen4).

Similar to the Poisson case, since $\{N(\zeta)\}_{\zeta\in Z}$ are memoryless, it is clear that the pair-type process $Q(\cdot)$ is a discrete-time Markov chain under $\mathbb{P}$, with state space $\mathcal{E}$ defined in (\ref{eeeh}). However, the jumps of $Q_{ij}(\cdot)$ are not necessarily of size $1$ because more than one animal can fire at each firing round.

\begin{proposition}\label{trafoggp}
Assume (Ber1)--(Ber3). Take a population with $x_1,\ldots,x_k$ females and $y_1,\ldots,y_k$ males of types $1,\ldots,k$, respectively, such that (\ref{destur}) holds. The transition kernel of the discrete-time Markov chain $Q(\cdot)$ has the following formula: for every $M,M'\in\mathcal{E}$,
\begin{align}
\rho(M,M') &:= \mathbb{P}\left(\left.Q(t+1) = M'\,\right|\,Q(t) = M\right)\nonumber\\
&\; = \frac{\Big(\prod_{i} (x_i - m_{i,\cdot})!\Big)\left(\prod_{j} (y_j - m_{\cdot,j})!\right)}{(n - m_{tot})!\left(\prod_{i,j}(m'_{ij} - m_{ij})!\right)}\sum_{\substack{M''\in\mathcal{E}':\\M''\ge M'}}\prod_{i,j}\frac{\left(\pi_{ij}\right)^{m'_{ij} - m_{ij}}\left(1-\pi_{ij}\right)^{m''_{ij}-m'_{ij}}}{(m''_{ij} - m'_{ij})!}\label{tayland}
\end{align}
if $M'\ge M$, and $0$ otherwise. Here,
\begin{equation}\label{piber}
\pi_{ij} := p_{ij}(\alpha_i + \beta_j - \alpha_i\beta_j)\in(0,1].
\end{equation}
\end{proposition}

\begin{proof}
By the Markov property, conditioning on $Q(t) = M$ is equivalent to taking $t=0$ and assuming that there are initially $x_i - m_{i,\cdot}$ single type-$i$ females and $y_j - m_{\cdot,j}$ single type-$j$ males for every $i,j\in\{1,\ldots,k\}$. We would like to find the distribution of $Q(1)$ under these assumptions. To this end, let $Q'(\cdot)$ be the pair-type process in the case of definite mating upon encounter.
As we have seen in the proof of Theorem \ref{kungfu}, $$Q'_{ij}(1)\,|\,Q'_{ij}(T)\sim B(Q'_{ij}(T), \lambda_{ij}(1)),$$ where $B$ denotes the binomial distribution and $$\lambda_{ij}(1) = F_i(1) + G_j(1) - F_i(1)G_j(1) = \alpha_i + \beta_j - \alpha_i\beta_j.$$
Now, recall that
$$Q_{ij}(1)\,|\,Q'_{ij}(1)\sim B(Q'_{ij}(1),p_{ij}).$$
Therefore, by thinning,
$$Q_{ij}(1)\,|\,Q'_{ij}(T)\sim B(Q'_{ij}(T),\pi_{ij}),$$ with $\pi_{ij} = p_{ij}\lambda_{ij}(1)$ as in (\ref{piber}), and (\ref{tayland}) is obtained by modifying the proof of Theorem \ref{kungfu} accordingly.
\end{proof}

\begin{remark}\label{degistirb}
It is evident from (\ref{tayland}) that the distribution of the pair-type process $Q(\cdot)$ depends on $P = (p_{ij})$, $\alpha_1,\ldots,\alpha_k$ and $\beta_1,\ldots,\beta_k$ only through the $k\times k$ matrix $\Pi := (\pi_{ij})$ whose entries are defined in (\ref{piber}). In other words, we have the freedom to change these parameters as long as $\Pi$ stays the same.
Therefore, under the assumption of Bernoulli firing times, the EM law of the species can be identified with $\Pi$.
\end{remark}

Since the pair-type process $Q(\cdot)$ is a time-homogeneous Markov chain with transition kernel $\rho$ as in (\ref{tayland}), the Kolmogorov backward equation for 
(\ref{tozert}) is
\begin{align}
&u_{ij}(t+1;x_1,\ldots,x_k;y_1,\ldots,y_k)\nonumber\\
&\quad = \sum_{M\in\mathcal{E}}\rho(0,M)\left[u_{ij}(t;x_1-m_{1,\cdot},\ldots,x_k-m_{k,\cdot};y_1-m_{\cdot,1},\ldots,y_k-m_{\cdot,k}) + m_{ij}\right].\label{yonetmp}
\end{align}
Note that (\ref{yonetmp}) is a closed system of recursive ordinary difference equations.

Similarly, using the strong Markov property, we can decompose (\ref{bayankut}) with respect to the possible values of $Q(t_1^*(\tau))$ and thereby obtain the following recursion:
\begin{align}
&u_{ij}^*(x_1,\ldots,x_k;y_1,\ldots,y_k)\nonumber\\
&\quad = \sum_{\substack{M\in\mathcal{E}:\\M\neq 0}}\frac{\rho(0,M)}{1 - \rho(0,0)}\left[u_{ij}^*(x_1-m_{1,\cdot},\ldots,x_k-m_{k,\cdot};y_1-m_{\cdot,1},\ldots,y_k-m_{\cdot,k}) + m_{ij}\right].\label{karaosmanp}
\end{align}

The recursive equations (\ref{yonetmp}) and (\ref{karaosmanp}) can be solved numerically, especially when $n$ is not too large. However, it is not clear if they can be solved analytically.

\subsection{Bernoulli fine balance and panmixia}

\begin{definition}\label{berfinedef}
The EM law $\Pi$ is said to satisfy the Bernoulli fine balance condition if
\begin{equation}\label{finep}
(1-\pi_{ij})(1-\pi_{i'j'}) = (1-\pi_{ij'})(1-\pi_{i'j})
\end{equation}
for every $i,j,i',j'\in\{1,\ldots,k\}$.
\end{definition}


\begin{proposition}\label{mulugp}
Assume (Ber1)--(Ber3). If the species is panmictic, then the EM law $\Pi$ satisfies the Bernoulli fine balance condition.
\end{proposition}

\begin{proof}
Assume that the species is panmictic. Observe that (\ref{finep}) is trivial if $i=i'$ or $j=j'$. Take a population with $n=2$ and $x_1 = x_2 = y_1 = y_2 = 1$. Use recursion (\ref{karaosmanp}) to write
\begin{align*}
\frac1{2} &= \frac{x_1y_1}{n} = u^*_{11}(x_1,\ldots,x_k;y_1,\ldots,y_k)\\
&= \sum_{\substack{M\in\mathcal{E}:\\M\neq 0}}\frac{\rho(0,M)}{1 - \rho(0,0)}\left[u_{11}^*(x_1-m_{1,\cdot},\ldots,x_k-m_{k,\cdot};y_1-m_{\cdot,1},\ldots,y_k-m_{\cdot,k}) + m_{11}\right]\\
&= \frac{\frac1{2}\left(\pi_{11}(1 - \pi_{22})[0+1] +  (1 - \pi_{11})\pi_{22}[1+0] + \pi_{11}\pi_{22}[0+1]\right)}{\frac1{2}\left(\pi_{11}(1 - \pi_{22}) +  (1 - \pi_{11})\pi_{22} + \pi_{11}\pi_{22}\right) + \frac1{2}\left(\pi_{12}(1 - \pi_{21}) +  (1 - \pi_{12})\pi_{21} + \pi_{12}\pi_{21}\right)}\\
&= \frac{\pi_{11} + \pi_{22} - \pi_{11}\pi_{22}}{(\pi_{11} + \pi_{22} - \pi_{11}\pi_{22}) + (\pi_{12} + \pi_{21} - \pi_{12}\pi_{21})}.
\end{align*}
This follows from computing $\rho(0,M)$ for every
$$M\in\mathcal{E} = \left\{\left(\begin{array}{ll}  1 & 0 \\ 0 & 0 \end{array}\right), \left(\begin{array}{ll}  0 & 0 \\ 0 & 1 \end{array}\right), \left(\begin{array}{ll}  1 & 0 \\ 0 & 1 \end{array}\right), \left(\begin{array}{ll}  0 & 1 \\ 0 & 0 \end{array}\right), \left(\begin{array}{ll}  0 & 0 \\ 1 & 0 \end{array}\right), \left(\begin{array}{ll}  0 & 1 \\ 1 & 0 \end{array}\right), \left(\begin{array}{ll}  0 & 0 \\ 0 & 0 \end{array}\right)\right\}$$
and observing that only the first three of these matrices contribute to the sum above.
Therefore, $\pi_{11} + \pi_{22} - \pi_{11}\pi_{22} = \pi_{12} + \pi_{21} - \pi_{12}\pi_{21}$. This is (\ref{finep}) with $i=j=1$ and $i'=j'=2$. By relabeling the types, the same argument works for any $i,j,i',j'\in\{1,\ldots,k\}$ such that $i\neq i'$ and $j\neq j'$, and we are done.
\end{proof}


\begin{lemma}\label{propjoep}
The EM law $\Pi$ satisfies the Bernoulli fine balance condition if and only if there exist $\bar\alpha_1,\ldots,\bar\alpha_k,\bar\beta_1,\ldots,\bar\beta_k\in[0,1]$ such that
\begin{equation}\label{thanksp}
1 - \pi_{ij} = (1 - \bar\alpha_i)(1 - \bar\beta_j)
\end{equation}
for every $i,j\in\{1,\ldots,k\}$.
\end{lemma}

\begin{proof}
Assume that (\ref{thanksp}) indeed holds with some $\bar\alpha_1,\ldots,\bar\alpha_k,\bar\beta_1,\ldots,\bar\beta_k\in[0,1]$. Then,
\begin{align*}
(1-\pi_{ij})(1-\pi_{i'j'}) &= (1 - \bar\alpha_i)(1 - \bar\beta_j)(1 - \bar\alpha_{i'})(1 - \bar\beta_{j'})\\
& = (1 - \bar\alpha_i)(1 - \bar\beta_{j'})(1 - \bar\alpha_{i'})(1 - \bar\beta_j) = (1-\pi_{ij'})(1-\pi_{i'j})
\end{align*}
for every  $i,j,i',j'\in\{1,\ldots,k\}$, i.e., the Bernoulli fine balance condition is satisfied.

Conversely, assume that the Bernoulli fine balance condition holds. If $\pi_{ij}=1$ for every $i,j\in\{1,\ldots,k\}$, then (\ref{thanksp}) holds with $\bar\alpha_i = \bar\beta_j = 1$. Otherwise, assume without loss of generality (i.e., up to relabeling the female and male types) that $\pi_{11} = \min\{\pi_{11},\ldots,\pi_{1k}\} < 1$. Define $$\bar\alpha_i := \pi_{i1}\in(0,1]\quad\mbox{and}\quad \bar\beta_j := 1 - \frac{1-\pi_{1j}}{1 - \pi_{11}}\in[0,1].$$
With this notation, we have
$$1-\pi_{ij} = \frac{(1-\pi_{i1})(1-\pi_{1j})}{1 - \pi_{11}} = (1 - \bar\alpha_i)(1 - \bar\beta_j),$$
where the first equality follows from (\ref{finep}). This concludes the proof.
\end{proof}

\begin{theorem}\label{yumosoglup}
Assume (Ber1)--(Ber3). Take a population with $x_1,\ldots,x_k$ females and $y_1,\ldots,y_k$ males of types $1,\ldots,k$, respectively, such that (\ref{destur}) holds. If the EM law $\Pi$ satisfies the Bernoulli fine balance condition, then
\begin{equation}\label{cikolata1p}
\mathbb{P}\left(Q(t) = M\right) = \frac{\Big(\prod_{i} x_i!\Big)\left(\prod_{j} y_j!\right)}{n!\left(\prod_{i,j}m_{ij}!\right)}\sum_{\substack{M'\in\mathcal{E}':\\M'\ge M}}\prod_{i,j}\frac{\left(1 - (1 - \pi_{ij})^t\right)^{m_{ij}}\left((1 - \pi_{ij})^t\right)^{m'_{ij}-m_{ij}}}{(m'_{ij} - m_{ij})!}
\end{equation}
for every $t\in\mathbb{N}\cup\{0\}$ and $M\in\mathcal{E}$. Similarly,
\begin{align}
\mathbb{P}(Q(T) = M) &= \frac{\Big(\prod_{i} x_i!\Big)\left(\prod_{j} y_j!\right)}{n!\left(\prod_{i,j}m_{ij}!\right)}\quad\text{for $M\in\mathcal{E}'$},\label{cikolata3p}\\
u^*_{ij}(x_1,\ldots,x_k;y_1,\ldots,y_k) &= \frac{x_iy_j}{n},\quad\mbox{and}\label{cikolata4p}\\
u_{ij}(t;x_1,\ldots,x_k;y_1,\ldots,y_k) &= \frac{x_iy_j}{n}\left(1 - (1 - \pi_{ij})^t\right)\label{cikolata2p}.
\end{align}
\end{theorem}

\begin{proof}
If the EM law satisfies the Bernoulli fine balance condition, then Lemma \ref{propjoep} implies that (\ref{thanksp}) holds with some $\bar\alpha_1,\ldots,\bar\alpha_k,\bar\beta_1,\ldots,\bar\beta_k\in[0,1]$, which is equivalent to
$$\pi_{ij} = \bar\alpha_i + \bar\beta_j - \bar\alpha_i\bar\beta_j.$$
In the light of Remark \ref{degistirb}, as far as the distribution of the pair-type process $Q(\cdot)$ is concerned, we can change (if necessary) $P = (p_{ij})$, $\alpha_1,\ldots,\alpha_k$ and $\beta_1,\ldots,\beta_k$, and
assume that $$p_{ij} = 1,\quad\alpha_i = \bar\alpha_i\quad\mbox{and}\quad\beta_j = \bar\beta_j$$
for every $i,j\in\{1,\ldots,k\}$. But then, we have definite mating upon encounter and Theorem \ref{kungfu} is applicable. Since $\tau_1(\zeta)$ is geometrically distributed with success probability $\bar\alpha_i$ (resp.\ $\bar\beta_j$) if $\zeta$ is a type-$i$ female (resp.\ type-$j$ male), we have
$$\lambda_{ij}(t) = 1 - (1 - F_i(t))(1 - G_j(t)) = 1 - (1 - \bar\alpha_i)^t(1 - \bar\beta_j)^t = 1 - (1 - \pi_{ij})^t$$
and this implies (\ref{cikolata1p}). Similarly, Corollaries \ref{korcount}, \ref{korpan} and \ref{aikido} give (\ref{cikolata3p}), (\ref{cikolata4p}) and (\ref{cikolata2p}), respectively.
\end{proof}

\begin{corollary}\label{karaksp}
Assume (Ber1)--(Ber3). The species is panmictic if and only if the EM law $\Pi$ satisfies the Bernoulli fine balance condition.
\end{corollary}

\begin{proof}
This is immediate from Proposition \ref{mulugp} and Theorem \ref{yumosoglup}.
\end{proof}

\subsection{Full characterization of panmixia, homogamy and heterogamy in the $2\times 2$ case}\label{bercharsec}

\begin{lemma}\label{propjekk}
For $k=2$, the EM law $\Pi$ satisfies $(1-\pi_{11})(1-\pi_{22}) > (1-\pi_{12})(1-\pi_{21})$ if and only if there exist $\bar\alpha_1,\bar\alpha_2,\bar\beta_1,\bar\beta_2\in[0,1]$ such that
\begin{align}
1-\pi_{11} &> (1 - \bar\alpha_1)(1 - \bar\beta_1),\quad 1 - \pi_{12} = (1 - \bar\alpha_1)(1 - \bar\beta_2),\label{thankkk}\\
1-\pi_{21} &= (1 - \bar\alpha_2)(1 - \bar\beta_1)\quad\mbox{and}\quad 1 - \pi_{22} > (1 - \bar\alpha_2)(1 - \bar\beta_2).\nonumber
\end{align}
\end{lemma}

\begin{proof}
Assume that (\ref{thankkk}) indeed holds with some $\bar\alpha_1,\bar\alpha_2,\bar\beta_1,\bar\beta_2\in[0,1]$. Then,
\begin{align*}
(1-\pi_{11})(1-\pi_{22}) &> (1 - \bar\alpha_1)(1 - \bar\beta_1)(1 - \bar\alpha_2)(1 - \bar\beta_2)\\
&= (1 - \bar\alpha_1)(1 - \bar\beta_2)(1 - \bar\alpha_2)(1 - \bar\beta_1) = (1-\pi_{12})(1-\pi_{21}).
\end{align*}

Conversely, if $(1-\pi_{11})(1-\pi_{22}) > (1-\pi_{12})(1-\pi_{21})$, then pick $\bar\Pi = (\bar\pi_{ij})$ such that
$$\pi_{11} < \bar\pi_{11},\quad\pi_{12} = \bar\pi_{12},\quad\pi_{21} = \bar\pi_{21},\quad\pi_{22} < \bar\pi_{22}\quad\mbox{and}\quad(1-\bar\pi_{11})(1-\bar\pi_{22}) = (1-\bar\pi_{12})(1-\bar\pi_{21}).$$
The desired result follows from Lemma \ref{propjoep} applied to $\bar\Pi$.
\end{proof}

\begin{theorem}\label{fullbull}
Assume (Ber1)--(Ber3). For $k=2$, the species is
\begin{itemize}
\item [(i)] heterogamous if\quad $(1-\pi_{11})(1-\pi_{22}) > (1-\pi_{12})(1-\pi_{21})$,
\item [(ii)] panmictic if\qquad\ \;$(1-\pi_{11})(1-\pi_{22}) = (1-\pi_{12})(1-\pi_{21})$, and
\item [(iii)] homogamous if\quad\ $(1-\pi_{11})(1-\pi_{22}) < (1-\pi_{12})(1-\pi_{21})$.
\end{itemize}
\end{theorem}

\begin{proof}
Let us prove part (i). If $(1-\pi_{11})(1-\pi_{22}) > (1-\pi_{12})(1-\pi_{21})$, then Lemma \ref{propjekk} implies that (\ref{thankkk}) holds with some $\bar\alpha_1,\bar\alpha_2,\bar\beta_1,\bar\beta_2\in[0,1]$. In the light of Remark \ref{degistirb}, as far as the distribution of the pair-type process $Q(\cdot)$ is concerned, we can change (if necessary) $P = (p_{ij})$, $\alpha_1,\alpha_2$ and $\beta_1,\beta_2$, and
assume that $$p_{11} < 1,\quad p_{12} = 1,\quad p_{21} = 1,\quad p_{22} < 1,\quad\alpha_i = \bar\alpha_i\quad\mbox{and}\quad\beta_j = \bar\beta_j$$
for every $i,j\in\{1,2\}$. But then, we have definite mating upon encounter for mixed-type temporary pairs only, and Theorem \ref{aaolurmu} is applicable. This concludes the proof of part (i). We have already shown part (ii) in Corollary \ref{karaksp}. Finally, part (iii) follows from part (i) by relabeling the male types.
\end{proof}

\section{Some concluding remarks, observations and open problems}\label{conicsection}

\subsubsection*{Main features of the model}

Gimelfarb \cite[p.\ 873]{Gim88a} says: ``Comparing first the applicability of the two encounter-mating models to real populations, it would seem that model 1 (individual encounters) is better suited for populations in which individuals are dispersed and search for potential mates individually, whereas model 2 (mass encounters) is better for populations in which individuals aggregate in search for potential mates, for example, mating swarms or leks. It is also possible that in some populations more than one encounter may take place at a time, and yet not all of the individuals available for mating participate in encounters simultaneously; thus, neither of the two models will describe such a population correctly." The SEM model unifies and generalizes the individual and mass EM models, and thereby removes the aforesaid limitation regarding their scope. Moreover, by tuning the parameters $\alpha_1,\ldots,\alpha_k$ and $\beta_1,\ldots,\beta_k$ in the cases of Poisson and Bernoulli firing times, possible differences in the vigor of animals searching for mates can be incorporated, cf.\ \cite[p. 870]{Gim88a}. However, it still remains to add other important features to the SEM model such as separation, births, deaths and offsprings in order to study the evolutionary aspects of mating behavior.

\subsubsection*{The key ideas}

We would like to recapitulate that the following observations have been of fundamental importance throughout this paper.
\begin{itemize}
\item [(i)] In the case of definite mating upon encounter, $\mathcal{L}(\cdot)$ is measurable with respect to $\mathcal{L}(T)$ and $\tau = (\tau_s(\zeta))_{s\ge1,\zeta\in Z}$ which are independent (see Theorem \ref{bekirog} and display (\ref{kompak})).
\item [(ii)] In the Poisson and Bernoulli cases, we have the freedom to change $P$, $\alpha_1,\ldots,\alpha_k$ and $\beta_1,\ldots,\beta_k$ as long as $\Pi$ stays the same (see Remarks \ref{degistira} and \ref{degistirb}).
\end{itemize}
Indeed, the first one provides us with a convenient way of decomposing $\mathcal{L}(\cdot)$ into independent components, and the second one (which we have been referring to as the change-of-parameters technique) allows us to reduce the model to definite mating upon encounter if the corresponding fine balance condition is satisfied.

%
%



%
%

\subsubsection*{On generic EM laws}

For Poisson firing times, the expectations $u_{ij}(t;x_1,\ldots,x_k;y_1,\ldots,y_k)$ and $u^*_{ij}(x_1,\ldots,x_k;y_1,\ldots,y_k)$, which were defined in (\ref{tozert}) and (\ref{bayankut}), solve the recursive equations (\ref{yonetm}) and (\ref{karaosman}), respectively. When the EM law $\Pi$ satisfies the Poisson fine balance condition, Theorem \ref{yumosoglu} provides formulas for (\ref{tozert}) and (\ref{bayankut}). Similarly, in the Bernoulli case, (\ref{tozert}) and (\ref{bayankut}) solve the recursive equations (\ref{yonetmp}) and (\ref{karaosmanp}), respectively, and Theorem \ref{yumosoglup} provides formulas for them when $\Pi$ satisfies the Bernoulli fine balance condition. One can also verify all of these statements by elementary combinatorics.

In either the Poisson or the Bernoulli case, if the corresponding fine balance condition is not satisfied, then we do not have formulas for (\ref{tozert}) and (\ref{bayankut}), i.e., it is not known whether the aforementioned recursive equations can be solved. These constitute interesting but rather difficult open problems. Alternatively, one can try to formulate and prove categorical results regarding (\ref{tozert}) and (\ref{bayankut}), such as Theorems \ref{fullpull} and \ref{fullbull} in which we use $\Pi$ to characterize heterogamy/panmixia/homogamy for $k=2$. It remains to figure out how these results can be generalized to $k\ge3$.

\subsubsection*{Three interesting observations regarding mating preferences}

First, as we have emphasized in Section \ref{literatur}, the terms panmixia, homogamy and heterogamy refer to the expected mating pattern and not to the mating preferences. Indeed, in the $k=2$ case with Poisson or Bernoulli firing times, homogamy does not imply that (say) type-$1$ females prefer mating with type-$1$ males. For example, take $$0 < p_{11} < p_{12} = p_{21} \le \frac1{2} < p_{22} = 1$$
with $\alpha_1 = \alpha_2 = 0$ and $\beta_1 = \beta_2 = 1$. Then, 
$$\pi_{11} + \pi_{22} > \pi_{12} + \pi_{21}\quad\mbox{and}\quad (1- \pi_{11})(1- \pi_{22}) < (1- \pi_{12})(1- \pi_{21}),$$
and we have homogamy both for Poisson and Bernoulli firing times by Theorems \ref{fullpull} and \ref{fullbull}. 

Second, in the $k=2$ case, there exist mating preference matrices that give heterogamy for Poisson firing times and homogamy for Bernoulli firing times (or vice versa). For example, suppose
$$p_{12} = p_{21} = \frac1{2} < \frac{3}{4} \le p_{22} < p_{11} + p_{22} < 1$$
with $\alpha_1 = \alpha_2 = 0$ and $\beta_1 = \beta_2 = 1$.
Then, $$\pi_{11} + \pi_{22} < \pi_{12} + \pi_{21}\quad\mbox{and}\quad (1- \pi_{11})(1- \pi_{22}) < (1- \pi_{12})(1- \pi_{21}),$$
and the claim follows from Theorems \ref{fullpull} and \ref{fullbull}.

Third, let us assume that there exists $c\in(0,1)$ such that $p_{ij} = c$ for every $i,j\in\{1,\ldots,k\}$. In words, animals have uniform preferences over types. This case is rather similar to definite mating upon encounter, i.e., when $c=1$. Therefore, one might predict that panmixia holds. This prediction is indeed true in the Poisson case with arbitrary $\alpha_1,\ldots,\alpha_k$ and $\beta_1,\ldots,\beta_k$ since we get $$\pi_{ij} = c(\alpha_i + \beta_j)$$
which satisfies the Poisson fine balance condition and Corollary \ref{karaks} applies. However, in the Bernoulli case with arbitrary $\alpha_1,\ldots,\alpha_k$ and $\beta_1,\ldots,\beta_k$, we get $$\pi_{ij} = c(\alpha_i + \beta_j - \alpha_i\beta_j) = c[1 - (1 - \alpha_i)(1 - \beta_j)].$$
It is easy to check that this EM law satisfies the Bernoulli fine balance condition if and only if $(\alpha_i - \alpha_{i'})(\beta_j - \beta_{j'}) = 0$ for every $i,j,i',j'\in\{1,\ldots,k\}$. Therefore, by Corollary \ref{karaksp}, panmixia is equivalent to having $\alpha_1 = \alpha_2 = \cdots = \alpha_k$ or $\beta_1 = \beta_2 = \cdots = \beta_k$. Moreover, in the $k=2$ case, it follows from Theorem \ref{fullbull} that there is homogamy when $(\alpha_1 - \alpha_2)(\beta_1 - \beta_2) < 0$ and heterogamy when $(\alpha_1 - \alpha_2)(\beta_1 - \beta_2) > 0$.

\subsubsection*{Asymptotics of the model}

As we have mentioned in Section \ref{literatur}, Gimelfarb carries out his analysis of the individual and mass EM models in \cite{Gim88a} after replacing all quantities such as $Q(T)$ with their expectations, and says that this approximation is justified by the law of large numbers when the population is large. Therefore, it is natural to ask if the law of large numbers indeed holds for the SEM model as $x_1,\ldots,x_k$ and $y_1,\ldots,y_k$ with (\ref{destur}) go to infinity in a suitable way. 
In a recent paper \cite{GunYil14b}, we establish the functional version of this limit theorem in the case of Poisson firing times, and describe the asymptotic mean of $Q(\cdot)$ as the solution of an ordinary differential equation system which we turn into Lotka-Volterra and replicator equations from population dynamics. We pay particular attention to the $k = 2$ case with $\pi_{12} = \pi_{21}$ and $x_1 = y_1$, and provide an explicit formula for the asymptotic (normalized) mating pattern without any fine balance assumption. We intend to do the same for Bernoulli firing times in the near future.

\section*{Acknowledgments}

We are indebted to A.\ Courtiol for introducing us to Gimelfarb's work on encounter-mating and for suggesting interesting problems. We also thank F.\ Rezakhanlou and P.\ Diaconis for valuable comments and discussions. O.\ G\"un gratefully acknowledges support by DFG SPP Priority Programme 1590 ``Probabilistic Structures in Evolution". A.\ Yilmaz is supported in part by European Union FP7 Marie Curie Career Integration Grant no.\ 322078.

\bibliographystyle{plain}
\bibliography{mating_references}

\end{document}